\documentclass[12pt]{amsart}
\usepackage[OT2,T1]{fontenc}
\DeclareSymbolFont{cyrletters}{OT2}{wncyr}{m}{n}
\DeclareMathSymbol{\Sha}{\mathalpha}{cyrletters}{"58}
\usepackage[all,cmtip]{xy}
\usepackage{enumerate, comment}
\usepackage{ amssymb, latexsym, amsmath}
\usepackage{fullpage}
\usepackage{url}
\usepackage{hyperref}
\usepackage{enumitem}
\usepackage{ crimson }
\usepackage{amsfonts}
\usepackage{tikz-cd}
\newcommand{\op}[1]{\operatorname{#1}}
\usepackage{amssymb}
\usepackage{amsthm}
\usepackage{amsmath}
\usepackage{amscd}
\usepackage[mathscr]{eucal}
\usepackage{indentfirst}
\usepackage{graphicx}
\usepackage{graphics}
\usepackage{pict2e}
\usepackage{epic}
\usepackage[OT2,T1]{fontenc}
\numberwithin{equation}{section}
\usepackage[margin=3.5cm]{geometry}
\usepackage{epstopdf}

\theoremstyle{plain}
\newtheorem{Th}{Theorem}[section]
\newtheorem{Lemma}[Th]{Lemma}
\newtheorem{Cor}[Th]{Corollary}
\newtheorem{Prop}[Th]{Proposition}

 \theoremstyle{definition}
\newtheorem{Def}[Th]{Definition}

\newtheorem{?}[Th]{Problem}

\newtheorem{hyp}{Hypothesis}

\newcommand*{\rom}[1]{\uppercase\expandafter{\romannumeral #1\relax}}
\newcommand{\Q}{\mathbb{Q}}
\newcommand{\Z}{\mathbb{Z}}

\newcommand{\F}{\mathbb{F}}

\newcommand{\Sel}{\operatorname{Sel}}

\DeclareSymbolFont{cyrletters}{OT2}{wncyr}{m}{n}

\newcommand\mtx[4] { \left( {\begin{array}{cc}
   #1 & #2 \\
   #3 & #4 \\
  \end{array} } \right)}
  
  \newcommand{\closure}[2][3]{%
{}\mkern#1mu\overline{\mkern-#1mu#2}}

\DeclareFontFamily{U}{wncy}{}
    \DeclareFontShape{U}{wncy}{m}{n}{<->wncyr10}{}
    \DeclareSymbolFont{mcy}{U}{wncy}{m}{n}
    \DeclareMathSymbol{\Sh}{\mathord}{mcy}{"58}

\begin{document}

\title{Euler Characteristics and their Congruences for Multi-signed Selmer Groups}

\author{Anwesh Ray}
\address{Department of Mathematics \\ University of British Columbia \\
  Vancouver BC, V6T 1Z2, Canada.} 
  \email{anweshray@math.ubc.ca}
\author{R. Sujatha}
\address{Department of Mathematics \\ University of British Columbia \\
  Vancouver BC, V6T 1Z2, Canada.} 
  \email{sujatha@math.ubc.ca} 
  
\maketitle

\begin{abstract}

The notion of the truncated Euler characteristic for Iwasawa modules is a generalization of the the usual Euler characteristic to the case when the cohomology groups are not finite. Let $p$ be an odd prime, $E_1$ and $E_2$ be elliptic curves over a number field $F$ with semistable reduction at all primes $v|p$ such that the $\op{Gal}(\closure{F}/F)$-modules $E_1[p]$ and $E_2[p]$ are irreducible and isomorphic. We compare the Iwasawa invariants of certain imprimitive multisigned Selmer groups of $E_1$ and $E_2$. Leveraging these results, congruence relations for the truncated Euler characteristics associated to these Selmer groups over certain $\Z_p^m$-extensions of $F$ are studied. Our results extend earlier congruence relations for elliptic curves over $\Q$ with good ordinary reduction at $p$.
    
\end{abstract}
\maketitle
\section{Introduction}
\par The Iwasawa theory of Galois representations, especially those arising from elliptic curves and Hecke eigencuspforms, affords deep insights into the arithmetic of such objects. Let $p$ be a fixed odd prime. Mazur \cite{mazur72} and Greenberg \cite{greenbergIT} initiated the Iwasawa theory of $p$-ordinary elliptic curves $E$ defined over $\Q$. The main object of study is the $p^{\infty}$-Selmer group over the cyclotomic $\Z_p$-extension, denoted $\op{Sel}(E/\Q^{\op{cyc}})$. Let $\Gamma$ denote the Galois group of the cyclotomic $\Z_p$-extension over $\Q$. When $E$ has good ordinary reduction at the prime $p$, it was conjectured by Mazur that the Selmer group $\op{Sel}(E/\Q^{\op{cyc}})$ is cotorsion as a module over the Iwasawa algebra $\Z_p[[\Gamma]]$. This is now a celebrated theorem of Kato \cite{katozeta}. The corresponding theory for $p$-supersingular elliptic curves was initiated by Perrin-Riou, in \cite{pr90} and has since gained considerable momentum, see \cite{kurihara1, kurihara2, kobayashi, pollack, iovita, leiloefflerzerbes, sprungpair,BDKaust,KO}. If $E$ has supersingular reduction at $p$, then $\op{Sel}(E/\Q^{\op{cyc}})$ is no longer $\Z_p[[\Gamma]]$-cotorsion. Kobayashi considered plus and minus Selmer groups, which were defined using plus and minus norm groups, that were introduced in \cite{pr90}. These signed Selmer groups are cotorsion over $\Z_p[[\Gamma]]$.
\par Suppose $E$ is an elliptic curve defined over a number field $F$. Let $\Sigma_p$ denote the set of primes of $F$ above $p$. In the case when $E$ has good reduction at all primes in $\Sigma_p$, and $F_v\simeq \Q_p$ for all $v\in \Sigma_p$ such that $E$ is supersingular at $v$, generalizations of the plus/minus Selmer groups have been studied in \cite{BDKmixedreduction, sujathafilipo}. The case where $E$ has semistable reduction at the primes in $\Sigma_p$ (i.e. good ordinary, multiplicative or good supersingular) has also been considered, see \cite{leilim,leilim2}. In this paper, we assume that $E$ has semistable reduction at the primes in $\Sigma_p$. In this \textit{mixed}-reduction setting, a plethora of Selmer groups depending on the reduction-types at the primes in $\Sigma_p$ can be defined. This collection of Selmer groups is referred to as the multi-signed Selmer groups. 
\par Denote by $\Sigma_{\op{ss}}(E)=\{\mathfrak{p}_1,\dots, \mathfrak{p}_d\}$ the set of primes of $F$ above $p$ at which $E$ has good supersingular reduction. Associated to each vector $\ddag\in \{+,-\}^d$, there is a multi-signed Selmer group $\op{Sel}^{\ddag}(E/F^{\op{cyc}})$, defined in section $\ref{section2}$. There are $2^d$ vectors $\ddag$ and hence $2^d$ Selmer groups to consider.
\par Two elliptic curves $E_1$ and $E_2$ over $F$ are said to be $p$-congruent if their associated residual representations are isomorphic, i.e., $E_1[p]$ and $E_2[p]$ are isomorphic as Galois modules. It is of particular interest in Iwasawa theory to study the relationship between Iwasawa invariants of the Selmer groups of $p$-congruent elliptic curves. Such investigations were initiated by Greenberg and Vatsal \cite{greenbergvatsal}, who considered $p$-congruent, $p$-ordinary elliptic curves $E_1$ and $E_2$ defined over $\Q$. They showed that the main conjecture is true for $E_1$ if and only if the main conjecture is true for $E_2$. To this end, they study the relationship between the algebraic and analytic Iwasawa invariants of $E_1$ and $E_2$. Their method involves examining the algebraic structure of certain imprimitive Selmer groups associated to $E_1$ and $E_2$ in comparing the Iwasawa-invariants. Let $\Sigma_0$ be the set of primes $v\nmid p$ of $F$ at which $E_1$ or $E_2$ has bad reduction. Greenberg and Vatsal compare the $\Sigma_0$-imprimitive Selmer groups of $E_1$ and $E_2$. Their results were generalized to the $p$-supersingular case for the plus and minus Selmer groups by B.D. Kim in \cite{BDKim} and Ponsinet \cite{ponsinet}. 
\par In this paper, the above mentioned results of Greenberg-Vatsal and B.D. Kim are generalized to the mixed reduction setting. It is shown that if the $\mu$-invariant of a certain imprimitive multi-signed Selmer group of $E_1$ is zero, then the $\mu$-invariant of the imprimitive multisigned Selmer group of $E_2$ is also zero. Further, it is shown that if these $\mu$-invariants are both zero, then the imprimitive $\lambda$-invariants match up. The reader may refer to Theorem $\ref{equalitymulambda}$ for the precise statement. In fact, we are able to do more than simply generalize the aforementioned results. Over the rational numbers, if Theorem $\ref{equalitymulambda}$ is specialized to the case of $p$-ordinary reduction (resp. $p$-supersingular reduction), we obtain refinements of the results Greenberg-Vatsal (resp. B.D. Kim, Ponsinet). This is due to the fact that we show that an \textit{optimal set} of primes $\Sigma_1$ may be chosen so that the imprimitive Iwasawa invariants match with respect to $\Sigma_1$. The set of primes $\Sigma_1$ is optimal in the sense that it is smaller than the full set of primes $\Sigma_0$ considered by Greenberg-Vatsal, B.D. Kim and Ponsinet, defined in the previous paragraph.

\par The Euler characteristic of the $p$-primary Selmer group an elliptic curve over $\Q$ with $p$-ordinary reduction may be defined when the Mordell Weil group is finite. Furthermore, it is of significant interest from the point of view of the $p$-adic Birch and Swinnerton-Dyer conjecture. The truncated Euler characteristic is a derived version of the usual Euler characteristic, which may be defined in the positive rank setting. Leveraging our results on $\mu$-invariants and imprimitive $\lambda$-invariants, we prove congruence relations for the truncated Euler characteristics of multi-signed Selmer groups of $p$-congruent elliptic curves. First, this is done over the cyclotomic $\Z_p$-extension of $F$ (cf. Theorem $\ref{gammacongruence}$). It is shown that the truncated Euler characteristic of $E_1$ is related to that of $E_2$ after one accounts for the auxiliary set of primes $\Sigma_1$. The smaller the set of primes $\Sigma_1$, the more refined the congruence relation between truncated Euler characteristics is. This is why it is of considerable importance that the set of primes $\Sigma_1$ be carefully chosen to be as small as possible. Next, we extend such reults to the multi-signed Selmer groups of certain $\Z_p^m$-extensions and we compare the Akashi series for $p$-congruent elliptic curves over such $\Z_p^m$-extensions (cf. Theorem $\ref{Gcongruence}$). Our results are informed by explicit examples which are listed in section $\ref{examples}$.
\newline\textit{Acknowledgements:} The second author gratefully acknowledges support from NSERC Discovery grant 2019-03987.
\section{Preliminaries}\label{section2}

\par Throughout, we fix an odd prime number $p$ and number field $F$. Let $F^{\op{cyc}}$ be the cyclotomic $\Z_p$-extension of $F$. Denote by $\Sigma_p$ the set of primes of $F$ above $p$. Let $E$ be an elliptic curve over $F$ and denote by $\Sigma_{\op{ss}}(E)=\{\mathfrak{p}_1,\dots, \mathfrak{p}_d\}$ the set of primes in $\Sigma_p$ at which $E$ has supersingular reduction. Let $\mathcal{F}_{\infty}/F$ be a Galois extension of $F$ containing $F^{\op{cyc}}$ such that $\op{Gal}(\mathcal{F}_{\infty}/F)\simeq \Z_p^{m}$ for an integer $m\geq 1$. Set $\op{G}:=\op{Gal}(\mathcal{F}_{\infty}/F)$, $H:=\op{Gal}(\mathcal{F}_{\infty}/F^{\op{cyc}})$ and let $\Gamma$ be the Galois group $\op{Gal}(F^{\op{cyc}}/F)$, which is identified with $\op{G}/\op{H}$. Associated to any pro-$p$ group $\mathcal{G}$, the Iwasawa algebra $\Z_p[[\mathcal{G}]]$ is the inverse limit $\varprojlim_U \Z_p[\mathcal{G}/U]$, where $U$ ranges over all open normal subgroups of $\mathcal{G}$. Set $\mathcal{F}_0:=F$ and for $n\geq 1$, let $\mathcal{F}_n$ denote the unique subextension $F\subseteq \mathcal{F}_n\subset \mathcal{F}_{\infty}$ such that $\op{Gal}(\mathcal{F}_n/F)\simeq (\Z/p^n\Z)^m$. We introduce the following hypothesis on the elliptic curve $E$.
\begin{hyp}\label{hypothesis}Throughout, $E$ is required to satisfy the following hypotheses:
\begin{enumerate}
\item $E$ has semistable reduction at all primes $v\in \Sigma_p$ (i.e., either good ordinary, good supersingular or bad multiplicative reduction).
\item The residual representation $E[p]$ is a $2$-dimensional $\F_p$-vector space which is irreducible as a $\op{Gal}(\closure{F}/F)$-module.
    \item For every prime $v\in \Sigma_{\op{ss}}(E)$, the completion $F_v$ is isomorphic to $\Q_p$.
    \item For $v\in \Sigma_{\op{ss}}(E)$, set $a_v(E):=1+p-\# \widetilde{E}_v(\F_p)$, where $\widetilde{E}_v$ is the reduction of $E$ at $v$. Assume that $a_v(E)=0$ for all $v\in \Sigma_{\op{ss}}(E)$.
\end{enumerate}
\end{hyp}

\begin{hyp}\label{Finfhyp} Let $E$ be an elliptic curve over $F$ satisifying Hypothesis $\eqref{hypothesis}$. We introduce a hypothesis on the tuple $(E, \mathcal{F}_{\infty})$.
\begin{enumerate}
    \item\label{onehyp2} The extension $\mathcal{F}_{\infty}/F$ is unramified away from a finite set of primes.
    \item\label{twohyp2} Let $\mathfrak{p}_i|p$ be a prime of $F$ at which $E$ has supersingular reduction. Let $n\geq 1$ and let $\eta$ be a prime of $\mathcal{F}_{n}$ above $\mathfrak{p}_i$. Then, $\mathcal{F}_{n,\eta}$ is isomorphic to $K_n$ for some finite unramified extension $K$ of $\Q_p$. Here, $K_n$ is the $n$-th layer in the cyclotomic $\Z_p$-extension of $K$.
\end{enumerate}
\end{hyp} The condition $\eqref{twohyp2}$ ensures that the signed norm groups of $E$ over $\mathcal{F}_{n,\eta}$ may be defined. Thus, the condition is required in defining multi-signed Selmer groups for the extension $\mathcal{F}_{\infty}/F$, see Definition $\ref{defmultiselmer}$. The reader is referred to example $\eqref{example2}$ wherein a $\Z_p^2$-extension $\mathcal{F}_{\infty}/F$ is considered for which condition $\eqref{twohyp2}$ is satisfied.
For an elliptic curve $E$ and $\Z_p^m$-extension $\mathcal{F}_{\infty}$ satisfying the above hypotheses, we introduce the multisigned Selmer groups. We introduce local conditions at each of the primes $\mathfrak{p}_i\in \Sigma_{\op{ss}}(E)$. By Hypothesis $\eqref{hypothesis}$, the local field $F_{\mathfrak{p}_i}$ is isomorphic to $\Q_p$. Denote by $\widehat{E}_{\mathfrak{p}_i}$ the formal group of $E$ over $F_{\mathfrak{p}_i}$. For any algebraic extension $L$ of $F_{\mathfrak{p}_i}\simeq \Q_p$, denote by $\widehat{E}_{\mathfrak{p}_i}(L):=\widehat{E}_{\mathfrak{p}_i}(\mathfrak{m}_L)$, where $\mathfrak{m}_L$ is the maximal ideal of $\mathcal{O}_L$. Let $K$ be a finite unramified extension of $\Q_p$ and $K^{\op{cyc}}$ the cyclotomic $\Z_p$-extension of $K$. Denote by $K_n$ the unique subextension of $K^{\op{cyc}}/K$ of degree $p^n$. Kobayashi introduced plus and minus norm groups:
\[ \widehat{E}_{\mathfrak{p}_i}^+(K_n) :=
\left\{P\in \widehat{E}_{\mathfrak{p}_i}(K_n)  \mid \op{tr}_{n/m+1} (P)\in \widehat{E}_{\mathfrak{p}_i}(K_m),   \text{ for }0\leq m < n\text{ and }m \text{ even }\right\},\]
\[\widehat{E}_{\mathfrak{p}_i}^-(K_n):=\left\{P\in \widehat{E}_{\mathfrak{p}_i}(K_n) \mid \op{tr}_{n/m+1} (P)\in \widehat{E}_{\mathfrak{p}_i}(K_m),\text{ for }0\leq m < n\text{ and }m \text{ odd }\right\},\]
where $\op{tr}_{n/m+1}:\widehat{E}_{\mathfrak{p}_i}(K_n)\rightarrow \widehat{E}_{\mathfrak{p}_i}(K_{m+1})$ denotes the trace map with respect to the formal group law on $\widehat{E}_{\mathfrak{p}_i}$. 
\par Let $\ddag=(\ddag_1,\dots, \ddag_d)\in \{+,-\}^d$ be a multi-signed vector, so that each component $\ddag_i$ is either a $+$ or $-$ sign. For each finite prime $v\notin \Sigma_{\op{ss}}(E)$, and each integer $n\geq 1$, set 
\begin{equation}\label{Hvdef}\mathcal{H}_v(\mathcal{F}_n, E[p^{\infty}]):=\prod_{\eta|v} \frac{H^1(\mathcal{F}_{n,\eta}, E[p^{\infty}])}{E(\mathcal{F}_{n,\eta})\otimes \Q_p/\Z_p},\end{equation}
where $\eta$ runs through the primes of $\mathcal{F}_n$ above $v$. For each prime $\mathfrak{p}_i\in \Sigma_{\op{ss}}(E)$, set 
\begin{equation}\label{Hvpmdef}\mathcal{H}_{\mathfrak{p}_i}^{\pm}(\mathcal{F}_n, E[p^{\infty}]):=\prod_{\eta|\mathfrak{p}_i} \frac{H^1(\mathcal{F}_{n,\eta}, E[p^{\infty}])}{\widehat{E}_{\mathfrak{p}_i}^{\pm}(\mathcal{F}_{n,\eta})\otimes \Q_p/\Z_p}.\end{equation}
Let $\eta$ be a prime of $\mathcal{F}_{\infty}$ above $\mathfrak{p}_i$. By assumption $\eqref{twohyp2}$, $\mathcal{F}_{n,\eta}$ is isomorphic to $K_n$ for some finite unramified extension $K$ over $\Q_p$. Thus the norm groups $\widehat{E}_{\mathfrak{p}_i}^{\pm} (\mathcal{F}_{n,\eta})$ are defined. We now come to the definition of the multisigned Selmer groups. 
\begin{Def}\label{defmultiselmer}Let $d$ be the cardinality of $\Sigma_{\op{ss}}(E)$ and let $\ddag=(\ddag_1,\dots, \ddag_d)$ be a multi-signed vector, so that each component $\ddag_i$ is either a $+$ sign or $-$ sign. Let $\Sigma$ be a finite primes of $F$ containing $\Sigma_p$ and the primes at which $E$ has bad reduction. Let $\mathfrak{p}_1,\dots, \mathfrak{p}_d$ be the primes in $\Sigma_{\op{ss}}(E)$. For $i=1,\dots, d$, let $\mathcal{H}_{\mathfrak{p}_i}^{\ddag_i}(\mathcal{F}_n, E[p^{\infty}])$ be the group defined above, see $\eqref{Hvpmdef}$. The multi-signed Selmer group $\op{Sel}^{\ddag}(E/\mathcal{F}_n)$ is the kernel of the map:
\[\Phi_{E,\mathcal{F}_n}^{\ddag}:H^1(\mathcal{F}_n,E[p^{\infty}])\rightarrow \prod_{v\in \Sigma\backslash \Sigma_{\op{ss}}(E)} \mathcal{H}_v(\mathcal{F}_n, E[p^{\infty}])\times \prod_{i=1}^d \mathcal{H}_{\mathfrak{p}_i}^{\ddag_i}(\mathcal{F}_n, E[p^{\infty}]).\] Let $\op{Sel}^{\ddag}(E/\mathcal{F}_{\infty})$ be the direct limit $\varinjlim_n \op{Sel}^{\ddag}(E/\mathcal{F}_n)$ and write $\op{X}^{\ddag}(E/\mathcal{F}_{\infty})$ for its Pontryagin-dual.
\end{Def}
\par For practical purposes, it is convenient to work with an alternative description of $\op{Sel}^{\ddag}(E/F^{\op{cyc}})$. For $v\in \Sigma_p\backslash \Sigma_{\op{ss}}(E)$ denote by $E_v$ the curve over the local field $F_v$. Since $E_v$ has good ordinary or multiplicative reduction, it fits in a short exact sequence of $\op{Gal}(\closure{F_v}/F_v)$-modules:
\begin{equation}\label{ses}0\rightarrow C_v\rightarrow E_v[p^{\infty}]\rightarrow D_v\rightarrow 0.\end{equation}
Here, $C_v$ and $D_v$ are both of corank one over $\Z_p$ with the property that $C_v$ is a divisible subgroup and $D_v$ is the maximal subgroup on which $\op{I}_v:=\op{Gal}(\closure{F_v}/F_v^{\op{nr}})$ acts via a finite order quotient. In fact, $D_v$ is specified as follows:
\begin{equation}\label{Ddef}D_v=\begin{cases}
\widetilde{E_v}[p^{\infty}]\text{ if }E_v \text{ has good ordinary reduction},\\
\Q_p/\Z_p(\phi)\text{ if }E_v \text{ has multiplicative reduction,}
\end{cases}\end{equation}
where $\phi$ is an unramified character, which is trivial if and only if $E_v$ has split multiplicative reduction. For further details, the reader is referred to the discussions in \cite[p. 150]{CG} and \cite[section 2]{greenbergvatsal}.
\par Note that above each prime $v$ of $F$, there are finitely many primes $\eta|v$ of $F^{\op{cyc}}$. For $\mathfrak{p}_i\in \Sigma_{\op{ss}}(E)$, set:
\[\mathcal{H}_{\mathfrak{p}_i}^{\pm}(F^{\op{cyc}}, E[p^{\infty}]):=\prod_{\eta|\mathfrak{p}_i} \frac{H^1(F_{\eta}^{\op{cyc}}, E[p^{\infty}])}{\widehat{E}^{\pm}(F^{\op{cyc}}_{\eta})\otimes \Q_p/\Z_p},\]where $\eta$ runs through the primes of $F^{\op{cyc}}$ above $\mathfrak{p}_i$. For $v\in \Sigma\backslash \Sigma_{\op{ss}}(E)$, set 
\begin{equation}\label{defhv}\mathcal{H}_{v}(F^{\op{cyc}}, E[p^{\infty}]):=\begin{cases}\prod_{\eta|v} \op{im}\left(H^1(F_{\eta}^{\op{cyc}}, E[p^{\infty}])\rightarrow H^1(\op{I}_{\eta}, D_v)\right)\text{ if }v\in \Sigma_p\backslash \Sigma_{\op{ss}}(E),\\
\prod_{\eta|v} H^1(F_{\eta}^{\op{cyc}}, E[p^{\infty}])\text{ if }v\in \Sigma\backslash \Sigma_{p},
\end{cases}\end{equation}
where $\op{I}_{\eta}$ is the inertia subgroup of $\op{Gal}(\closure{F_{\eta}^{\op{cyc}}}/F_{\eta}^{\op{cyc}})$. The Selmer group $\op{Sel}^{\ddag}(E/F^{\op{cyc}})$ coincides with the kernel of the restriction map

\[\Phi_{E,F^{\op{cyc}}}^{\ddag}:H^1(F^{\op{cyc}},E[p^{\infty}])\rightarrow \prod_{v\in \Sigma\backslash \Sigma_{\op{ss}}(E)} \mathcal{H}_v(F^{\op{cyc}}, E[p^{\infty}])\times \prod_{i=1}^d \mathcal{H}_{\mathfrak{p}_i}^{\ddag_i}(F^{\op{cyc}}, E[p^{\infty}]) .\]For further details, see \cite[p.32,42]{greenbergvatsal} and the discussion in \cite[section 5]{greenbergIT}. Let $\Sigma_0$ be a finite set of primes of $F$ which does not contain any prime $v\in \Sigma_p$.
The $\Sigma_0$-imprimitive Selmer group $\op{Sel}^{\Sigma_0,\ddag}(E/F^{\op{cyc}})$ is the kernel of the restriction map

\[\Phi_{E,F^{\op{cyc}}}^{\Sigma_0,\ddag}:H^1(F^{\op{cyc}},E[p^{\infty}])\rightarrow \prod_{v\in \Sigma\backslash (\Sigma_{\op{ss}}(E)\cup \Sigma_0)} \mathcal{H}_v(F^{\op{cyc}}, E[p^{\infty}])\times \prod_{i=1}^d \mathcal{H}_{\mathfrak{p}_i}^{\ddag_i}(F^{\op{cyc}}, E[p^{\infty}]) .\] The Pontryagin dual of $\op{Sel}^{\Sigma_0, \ddag}(E/F^{\op{cyc}})$ is denoted by $\op{X}^{\Sigma_0, \ddag}(E/F^{\op{cyc}})$.
The next Proposition follows from \cite[Proposition 2.1]{greenbergvatsal} and \cite[Proposition 4.6]{leilim2}.
\begin{Prop}\label{restrictionmapsurj}
Assume that the Selmer group $\op{Sel}^{\Sigma_0, \ddag}(E/F^{\op{cyc}})$ is cotorsion as a $\Z_p[[\Gamma]]$-module. Then, the maps $\Phi_{E,F^{\op{cyc}}}^{\ddag}$ and $\Phi_{E,F^{\op{cyc}}}^{\Sigma_0,\ddag}$ are surjective.
\end{Prop}
Thus we have a short exact sequence relating the $\Sigma_0$-imprimitive Selmer group $\op{Sel}^{\Sigma_0,\ddag}(E/F^{\op{cyc}})$ with the Selmer group $\op{Sel}^{\ddag}(E/F^{\op{cyc}})$
\[0\rightarrow \op{Sel}^{\ddag}(E/F^{\op{cyc}})\rightarrow \op{Sel}^{\Sigma_0,\ddag}(E/F^{\op{cyc}})\xrightarrow{\Phi_{E,F^{\op{cyc}}}^{\ddag, \Sigma_0}}\prod_{v\in \Sigma_0} \mathcal{H}_v(F^{\op{cyc}},E[p^{\infty}])\rightarrow 0.\]
 Choose a topological generator $\gamma$ of $\Gamma$ and identify $\gamma-1$ with $T$ in choosing an isomorphism $\Z_p[[\Gamma]]\simeq \Z_p[[T]]$. A polynomial $f(T)\in \Z_p[T]$ is said to be \textit{distinguished} if it is a monic polynomial and all non-leading coefficients are divisble by $p$. By the structure theorem of finitely generated torsion $\Z_p[[T]]$-modules, there is a pseudo-isomorphism 
\[\op{X}^{\ddag}(E/F^{cyc})\sim \left(\bigoplus_{i=1}^n \Z_p[[T]]/(f_i(T))\right)\oplus \left(\bigoplus_{j=1}^m \Z_p[[T]]/(p^{\mu_j}) \right).\]For, $i=1,\dots, n$, the elements $f_i(T)$ above are distinguished polynomials and the product 
$f_{E}^{\ddag}(T):=p^{\sum_j \mu_j} \prod_{i} f_i(T)$ is called the characteristic polynomial. The signed Iwasawa invariants are defined by $\lambda_E^{\ddag}:=\op{deg} f_E^{\ddag}(T)$ and $\mu_E^{\ddag}:=\sum_{j=1}^m \mu_j$. Denote by $f_E^{\Sigma_0,\ddag}(T)$, $\lambda_E^{\Sigma_0,\ddag}$ and $\mu_E^{\Sigma_0,\ddag}$ the characteristic polynomial, $\lambda$-invariant and $\mu$-invariant of the imprimitive Selmer group $\op{X}^{\Sigma_0,\ddag}(E/F^{\op{cyc}})$.

\section{The Truncated Euler Characteristic}

\par In this section, we recall the notion of the truncated Euler characteristic, which is a generalization of the usual Euler characteristic. We then discuss explicit formulas for the truncated Euler characteristic, as predicted by the $p$-adic Birch and Swinnerton-Dyer conjecture. For a more detailed exposition, the reader may refer to \cite[section 3]{CSSLinks} and \cite{Zerbes}. For a discrete $p$-primary cofinitely generated $\Z_p[[\Gamma]]$-module $M$ for which the cohomology groups $H^i(\Gamma, M)$ have finite order, the Euler characteristic is defined to be the quotient \[\chi(\Gamma, M):=\frac{\# H^0(\Gamma, M)}{\# H^1(\Gamma, M)}.\] For an elliptic curve $E$ with potentially good ordinary reduction at all primes above $p$, the Euler characteristic of the Selmer group is defined, provided the Mordell-Weil group is finite. This definition does not extend to the case when the Mordell-Weil group of $E$ is infinite. The natural substitute is the \textit{truncated Euler characteristic}. 
\begin{Def}
Let $M$ be a discrete $p$-primary $\Gamma$-module let $\phi_M$ be the natural map 
\[\phi_M: H^0\left(\Gamma, M\right)=M^{\Gamma}\rightarrow M_{\Gamma}\simeq H^1\left(\Gamma,M\right)\]
for which $\phi_M(x)$ is the residue class of $x$ in $M_{\Gamma}$. The module $M$ is said to have finite truncated Euler-characteristic if both $\operatorname{ker}(\phi_M)$ and $\operatorname{cok}(\phi_M)$ are finite. In this case, the truncated Euler characteristic $\chi_t(\Gamma, M)$ is defined by \[\chi_t\left(\Gamma, M\right):=\frac{\#\op{ker}(\phi_M)}{\#\op{cok}(\phi_M)}.\]
\end{Def}
\begin{Lemma}
Assume that $M$ is a discrete $\Z_p[[\Gamma]]$-module. Assume that $M$ is cofinitely generated and cotorsion as a $\Z_p[[\Gamma]]$-module. The Euler characteristic $\chi(\Gamma, M)$ is defined if and only if $r_M=0$.
\end{Lemma}

\begin{proof}
Let $X$ be the Pontryagin dual of $M$. The invariant submodule $M^{\Gamma}$ is dual to $X_{\Gamma}$ and the quotient $M_{\Gamma}$ is dual to $X^{\Gamma}$. By an application of the structure theorem for finitely generated torsion $\Z_p[[\Gamma]]$-modules, there is a pseudoisomorphism
\[X\sim\bigoplus_{i=1}^m \Z_p[[\Gamma]]/(f_i(T))\]for some elements $f_i(T)\in \Z_p[[\Gamma]]$. The module $X_{\Gamma}$ may be identified with $X/TX$ and $X^{\Gamma}$ is the kernel of the multiplication by $T$ map $\times T: X\rightarrow X$. It is easy to see that the groups $X_{\Gamma}$ and $X^{\Gamma}$ are finite precisely when $T\nmid f_i(T)$ for $i=1,\dots, m$. Therefore, $X_{\Gamma}$ and $X^{\Gamma}$ are finite precisely when $r_M=0$.
\end{proof}
The following lemma is a criterion for the truncated Euler characteristic to be well defined.
 \begin{Lemma}\label{truncdefined}
Assume that $M$ is a $p$-primary, discrete $\Z_p[[\Gamma]]$-module. Let $X$ be the Pontryagin dual of $M$. Assume that $X$ is a finitely generated torsion $\Z_p[[\Gamma]]$-module. Denote by $X[p^{\infty}]$ the $p^{\infty}$-torsion submodule of $X$. Let $f_1(T), \dots, f_n(T)$ be distinguished polynomials such that $X/X[p^{\infty}]$ is pseudo-isomorphic to $\bigoplus_{i=1}^n \Z_p[[T]]/(f_i(T))$. Suppose that none of the polynomials $f_i(T)$ is divisible by $T^2$. Then, the kernel and cokernel of $\phi_M$ are finite and the truncated Euler characteristic $\chi_t(\Gamma, M)$ is defined. In particular, the truncated Euler characteristic $\chi_t(\Gamma, M)$ is defined when $r_M\leq 1$.
\end{Lemma}
\begin{proof}
The assertion of the lemma follows from the proof of \cite[Lemma 2.11]{Zerbes}.
\end{proof}
 Let $f_M(T)$ be the characteristic polynomial of the Pontryagin dual of $M$ and write $f_M(T)= T^{r_M} g_M(T)$, where $g_M(0)\neq 0$. Let $|\cdot|_p$ denote the absolute value on $\Q_p$ normalized by $|p|_p=p^{-1}$. When both $\op{ker}(\phi_M)$ and $\op{cok}(\phi_M)$ are finite, the truncated Euler characteristic $\chi_t(\Gamma, M)$ is related to the quantity $|g_M(0)|_p$.
\begin{Lemma}\label{TECbasiclemma}
Let $M$ be a discrete $\Z_p[[\Gamma]]$ module which is cofinitely generated and cotorsion. If the kernel and cokernel of $\phi_M$ are finite, then \[\chi_t(\Gamma, M)=|g_M(0)|_p^{-1},\] and further, $r_M=\op{cork}_{\Z_p} M^{\Gamma}=\op{cork}_{\Z_p} M_{\Gamma}$.
\end{Lemma}
\begin{proof}
The assertion follows from \cite[Lemma 2.11]{Zerbes}.
\end{proof}
 Evidently, it follows that $\chi_t(\Gamma, M)=p^N$, where $N\in \Z_{\geq 0}$. Let $\mu_M$ (resp. $\lambda_M$) denote its $\mu$-invariant (resp. $\lambda$-invariant) of $M^{\vee}$ as a $\Z_p[[\Gamma]]$-module.
 \begin{Lemma}\label{TECmulambda}
Let $M$ be a cofinitely generated cotorsion $\Z_p[[\Gamma]]$-module such that $\phi_M:M^{\Gamma}\rightarrow M_{\Gamma}$ has finite kernel and cokernel. Then, the following are equivalent:
\begin{enumerate}[label=(\alph*)]
\item\label{TECmulambdac1} $\chi_t(\Gamma, M)=1$,
\item\label{TECmulambdac2} $\mu_M=0$ and $\lambda_M=r_M$.
\end{enumerate}

\end{Lemma}

\begin{proof}
\par Suppose that $\chi_t(\Gamma, M)=1$. Recall that $g_M(T)$ is a polynomial such that $f_M(T)=T^{r_M} g_M(T)$ and $T\nmid g_M(T)$. By Lemma $\ref{TECbasiclemma}$, \[|g_M(0)|_p^{-1}=\chi_t(\Gamma, M)=1.\]As a result, $f_M(T)$ and $g_M(T)$ are distinguished polynomials. Since $g_M(0)$ is a unit, it follows that $g_M(T)$ is a unit. Since $g_M(T)$ is a distinguished polynomial, it follows that\[g_M(T)=1\text{ and }f_M(T)=T^{r_M}.\] As a result, $\mu_M=1$ and $\lambda_M=\op{deg} f_M(T)=r_M$. 
\par Conversely, suppose that $\mu_M=0$ and $\lambda_M=r_M$. Since $\mu_M=0$, it follows that $f_M(T)$ and $g_M(T)$ are distinguished polynomials. The degree of $f_M(T)$ is $\lambda_M=r_M$, it follows that $g_M(T)$ is a constant polynomial and hence, $g_M(T)=1$. By Lemma $\ref{TECbasiclemma}$, \[\chi_t(\Gamma, M)=|g_M(0)|_p^{-1}=1.\]\end{proof}

 Let $r_E^{\ddag}$ denote the order of vanishing of $f_E^{\ddag}(T)$ at $T=0$, and write 
 \[f_E^{\ddag}(T)=p^{r_E^{\ddag}} g_E^{\ddag}(T).\]Note that $g_E^{\ddag}(0)\neq 0$. According to Lemma $\ref{TECbasiclemma}$, the truncated Euler characteristic $\chi_t^{\ddag}(\Gamma, E):=\chi_t(\Gamma, \op{Sel}^{\ddag}(E/F^{\op{cyc}}))$ is determined by the constant term of $g_E^{\ddag}(T)$. By Lemma $\ref{TECmulambda}$, if the truncated Euler characteristic is defined, then $\chi_t^{\ddag}(\Gamma, E)=1$ if and only if $\mu_E^{\ddag}=0$ and $\lambda_E^{\ddag}=r_E^{\ddag}$. 
 
\par We next discuss the $p$-adic Birch and Swinnerton-Dyer conjecture and its relationship with explicit formulas for truncated Euler characteristics. Note that there are formulations of the $p$-adic Birch and Swinnerton-Dyer conjecture in very general contexts (see for instance \cite{rubin, designi}). For ease of exposition, we restrict ourselves to the case where the elliptic curves $E$ are defined over $\Q$. For elliptic curves with good ordinary or multiplicative reduction, the $p$-adic Birch and Swinnerton-Dyer conjecture in its current form was formulated by Mazur, Tate and Teitelbaum \cite[p. 38]{MTT}. This is a $p$-adic analog of the classical Birch and Swinnerton-Dyer conjecture which predicts the order of vanishing of the Mazur and Swinnerton-Dyer $p$-adic L-function $\mathcal{L}(E/\Q,T)$ at $T=0$, and postulates an explicit formula for the leading term (see also \cite{bernardietal1}). When $E$ has good supersingular reduction at $p$, a version of the $p$-adic Birch and Swinnerton-Dyer conjecture for plus and minus $p$-adic L-functions was formulated by Sprung \cite{sprung}. The conjecture of \textit{loc. cit.} is equivalent to that of Bernardi and Perrin-Riou \cite{bernardietal1}. Lemma $\ref{TECbasiclemma}$ asserts that the truncated Euler-characteristic is related to the leading coefficient of the characteristic element of the Selmer group. The main conjecture and the $p$-adic Birch and Swinnerton-Dyer conjecture together predict precise formulas for the truncated Euler characteristic.

 \par Assume that $E$ has either good ordinary or multiplicative reduction at $p$. When $E$ has split multiplicative reduction at $p$, set $\mathcal{L}_p(E)$ to denote the $\mathcal{L}$-invariant associated to the Galois representation on the $p$-adic Tate module of $E$ (see \cite[p. 407]{greenbergvatsal}). The $p$-adic height pairing (cf. \cite{heightpairings1} and \cite{heightpairings2}) is a $p$-adic analog of the usual height pairing. This pairing is conjectured to be non-degenerate (cf. \cite{heightpairings2}) and the $p$-adic regulator $R_p(E/\Q)$ is defined to be the determinant of this pairing. Let $\kappa$ denote the $p$-adic cyclotomic character. Fix a branch of the $p$-adic logarithm and set $\mathcal{R}_{\gamma}(E/\Q)$ to denote the normalized height pairing $(\log_p\kappa(\gamma))^{-r} R_p(E/\Q)$, where $r$ denotes the rank of the Mordell Weil group $E(\Q)$. Let $E_0(\Q_l)\subset E(\Q_l)$ be the subgroup of $l$-adic points with non-singular reduction modulo $l$. Denote by $\mathcal{\tau}(E)$ the Tamagawa product $\prod_l c_l$, where $c_l$ is the index of $E_0(\Q_l)$ in $E(\Q_l)$. Let $a_l(E)$ be the $l$-th coefficient of the normalized eigenform associated to $E$. For $p$-adic numbers $a$ and $b$, write $a\sim b$ if $a=ub$ for a $p$-adic unit $u$. Let $r_E^{\op{an}}$ denote the order of vanishing of $\mathcal{L}(E/\Q,T)$ at $T=0$.
\begin{Th}\label{pbsdconj}(Perrin-Riou \cite{perrinriou}, Schneider \cite{heightpairings2}, Jones \cite{Jonesmultiplicative}) Suppose that $E$ has either good ordinary reduction or multiplicative reduction at $p$. Assume that:
\begin{enumerate}
    \item the truncated Euler-characteristic $\chi_t(\Gamma, E)$ is defined,
    \item the $p$-adic regulator $R_p(E/\Q)$ is non-zero,
    \item $\Sh (E/\Q)[p^{\infty}]$ has finite cardinality.
\end{enumerate}
Then, the following assertions are true.
\begin{enumerate}[label=(\alph*)]
    \item If $E$ has either good ordinary reduction or non-split multiplicative reduction at $p$, then the analytic rank $r_E^{\op{an}}$ is equal to $r$. If $E$ has split multiplicative reduction at $p$, then the analytic rank $r_E^{\op{an}}$ is equal to $r+1$.
    \item If $E$ has either good ordinary reduction at $p$ or non-split multiplicative reduction at $p$, then
    \[\chi_t(\Gamma, E)\sim  \epsilon_p(E)\times \frac{\mathcal{R}_{\gamma}(E/\Q)\times \#(\Sh (E/\Q)[p^{\infty}])\times \tau(E)}{\#(E(\Q)_{\op{tors}})^2}.\]Here $\epsilon_p(E)$ is set to be $(1-\frac{1}{\alpha})^{s}$, where $\alpha$ is the unit root of the Hecke polynomial $X^2-a_p(E) X+p$ and
    \[s=\begin{cases}s=2\text{ if }E\text{ has good ordinary reduction at }p,\\
    s=1\text{ if }E\text{ has non-split multiplicative reduction at $p$}.
    \end{cases}\]
    \item If $E$ has split-multiplicative reduction, the $\mathcal{L}$-invariant $\mathcal{L}_p(E)$ plays a role and we have:
    \[\chi_t(\Gamma, E)\sim  \frac{\mathcal{L}_p(E)}{\log_p(\kappa(\gamma))}\times \frac{\mathcal{R}_{\gamma}(E/\Q)\times \#(\Sh (E/\Q)[p^{\infty}])\times \tau(E)}{\#(E(\Q)_{\op{tors}})^2}.\]
\end{enumerate}

\end{Th}

\section{Iwasawa Invariants of Congruent Elliptic curves}
\par In this section, we show that the \textit{imprimitive} Iwasawa-invariants associated to congruent elliptic curves satisfy certain relations. Throughout, $E$ is an elliptic curve over $F$ which satisfies Hypothesis $\eqref{hypothesis}$. Denote by $T_p(E):=\varprojlim_n E[p^n]$ the $p$-adic Tate-module equipped with natural $\op{Gal}(\closure{F}/F)$ action and set $V_p(E):=T_p (E)\otimes \Q_p$. At a prime $v\in \Sigma_p\backslash \Sigma_{\op{ss}}(E)$, recall that there are corank one $\Z_p$-modules which fit into a short exact sequence
\[0\rightarrow C_v\rightarrow E_v[p^{\infty}]\rightarrow D_v\rightarrow 0.\] When $E$ has good ordinary reduction at $v$, the quotient $D_v$ may be identified with $\widetilde{E}_{v}[p^{\infty}]$, where $\widetilde{E}_v$ is the reduction of $E$ at $v$. On the other hand, when $E$ has multiplicative reduction at $v$, identify $D_v$ with a twist of $\Q_p/\Z_p$ by an unramified character.
\begin{hyp}\label{hypothesis1}
Let $E$ be an elliptic curve over $F$ which satisfies Hypothesis $\eqref{hypothesis}$. Let $\Sigma_{\op{ss}}(E)=\{\mathfrak{p}_1,\dots, \mathfrak{p}_d\}$ be the set of primes $v|p$ of $F$ at which $E$ has supersingular reduction. Let $\ddag\in \{+,-\}^d$ be a signed vector. Then $(E, \ddag)$ satisfies the following conditions:
\begin{enumerate}
    \item The Selmer group $\op{Sel}^{\ddag}(E/F^{\op{cyc}})$ is $\Z_p[[\Gamma]]$-cotorsion.
    \item The truncated Euler characteristic $\chi_t^{\ddag}(\Gamma, E)$ is defined. In other words, the natural map
    $\phi_E^{\ddag}: \op{Sel}^{\ddag}(\Gamma, E)^{\Gamma}\rightarrow  \op{Sel}^{\ddag}(\Gamma, E)_{\Gamma}$  has finite kernel and cokernel.
\end{enumerate}

\end{hyp}
Let $E_1$ and $E_2$ be $p$-congruent elliptic curves over $F$. It follows from the proof of \cite[Proposition 3.9]{sujathafilipo} that $\Sigma_{\op{ss}}(E_1)$ is equal to $\Sigma_{\op{ss}}(E_2)$. Set $\Sigma_{\op{ss}}=\{\mathfrak{p}_1,\dots, \mathfrak{p}_d\}$ to denote the set of supersingular primes $v|p$ of $E_1$ and $E_2$, and let $\ddag\in \{+,-\}^{d}$ be a signed vector. We introduce the following hypothesis on the triple $(E_1,E_2,\ddag)$. \begin{hyp}\label{hypothesis2}
Hypothesis $\eqref{hypothesis1}$ holds for both $(E_1,\ddag)$ and $(E_2,\ddag)$.
\end{hyp}
We now define the imprimitive Selmer group associated to the residual Galois representation $E_i[p]$. Let $\Sigma$ be a set of finite primes of $F$ containing $\Sigma_p$ and the primes at which $E_1$ or $E_2$ has bad reduction. For a prime $v\in \Sigma_p$ let $\eta_v$ be the unique prime of $F^{\op{cyc}}$ above $v$. Let $\op{I}_{\eta_v}$ be the inertia subgroup of $\op{Gal}(\closure{F_{\eta_v}^{\op{cyc}}}/F_{\eta_v}^{\op{cyc}})$. For a prime $v\in \Sigma_{\op{ss}}$, and $i=1,2$, define \[\mathcal{H}_v^{\pm}(F^{\op{cyc}},E_i[p]):=\frac{H^1(F^{\op{cyc}}_{\eta_v}, E_i[p])}{\widehat{E}_i^{\pm}(F^{\op{cyc}}_{\eta_v})/p\widehat{E}_i^{\pm}(F^{\op{cyc}}_{\eta_v})}.\] For $v\in \Sigma\backslash \Sigma_{\op{ss}}$, set
\[\mathcal{H}_v(F^{\op{cyc}},E_i[p]):=\begin{cases}\prod_{\eta| v} H^1(F^{\op{cyc}}_{\eta}, E_i[p])\text{ if }v\in \Sigma\backslash \Sigma_p,\\
H^1(\op{I}_{\eta_v}, D_v(E_i)[p])\text{ if }v\in \Sigma_{p}\backslash \Sigma_{\op{ss}}.
\end{cases}\]
For an elliptic curve $E$ over $F$, denote by $\mathcal{N}_E$ the conductor of $E$ and $\closure{\mathcal{N}}_E$ the prime to $p$ part of the Artin conductor of the residual representation $E[p]$. Let $v\nmid p$ be a finite prime of $F$. Note that $v$ divides $\mathcal{N}_E$ (resp. $\closure{\mathcal{N}}_E$) is and only if it is a bad reduction prime of $E$ (resp. the residual Galois representation $E[p]$ is ramified at $v$). To ease notation, let $\mathcal{N}_1$ and $\mathcal{N}_2$ denote the conductors of $E_1$ and $E_2$ respectively. Denote by $\closure{\mathcal{N}}$ the prime to $p$ part of the Artin conductor of $E_1[p]$. Note that since $E_1[p]$ is isomorphic to $E_2[p]$, $\closure{\mathcal{N}}$ is the conductor of $E_2[p]$.
\begin{Def}\label{sigmazerodef}
Let $E$ be an elliptic curve over $F$ and $\Sigma_1(E)$ denote the subset of primes of $F$ such that $(\mathrm{i})$ $v\nmid p$,
    $(\mathrm{ii})$ $v|(\mathcal{N}_E/\closure{\mathcal{N}}_E)$,
    $(\mathrm{iii})$ if $\mu_p$ is contained in $F_v$, then $E$ has split multiplicative reduction at $v$. In the case $p=3$, set $\Sigma_1(E)$ to be the set of primes of $F$ such that $(\mathrm{i})$ $v\nmid p$,
    $(\mathrm{ii})$ $v|(\mathcal{N}_E/\closure{\mathcal{N}}_E)$.
\end{Def} The $\Sigma_1$-imprimitive mod-$p$ Selmer group $\op{Sel}^{\Sigma_1,\ddag}(E_i[p]/F^{\op{cyc}})$ is defined to be the kernel of the restriction map:
\[\closure{\Phi}_{E_i}^{\Sigma_1,\ddag}:H^1(F_{\Sigma}/F^{\op{cyc}}, E_i[p])\rightarrow \prod_{\Sigma/(\Sigma_{\op{ss}}\cup \Sigma_1)} \mathcal{H}_v(F^{\op{cyc}},E_i[p])\times \prod_{j=1}^d\mathcal{H}_{\mathfrak{p}_j}^{\ddag_j}(F^{\op{cyc}},E_i[p]).\]
\begin{Prop}\label{prop42}
Let $E_1$ and $E_2$ be elliptic curves over $F$ which are $p$-congruent. Let $\ddag$ be a signed vector and assume that $(E_1,E_2,\ddag)$ satisfies Hypothesis $\eqref{hypothesis2}$. Let $\Sigma_1$ be the set of primes as in Definition $\ref{sigmazerodef}$. Then the isomorphism $E_1[p]\simeq E_2[p]$ induces an isomorphism of Selmer groups $\op{Sel}^{\Sigma_1,\ddag}(E_1[p]/F^{\op{cyc}})\simeq \op{Sel}^{\Sigma_1,\ddag}(E_2[p]/F^{\op{cyc}})$.
\end{Prop}
\begin{proof}
Let $\Phi:E_1[p]\xrightarrow{\sim} E_2[p]$ be a choice of isomorphism of Galois modules. Clearly, $\Phi$ induces an isomorphism $H^1(F_{\Sigma}/F^{\op{cyc}}, E_1[p])\xrightarrow{\sim} H^1(F_{\Sigma}/F^{\op{cyc}}, E_2[p])$. It suffices to show that for $v\in \Sigma$, the isomorphism $\Phi:E_1[p]\xrightarrow{\sim} E_2[p]$ induces an isomorphism 
\[\mathcal{H}_v(F^{\op{cyc}},E_1[p])\xrightarrow{\sim} \mathcal{H}_v(F^{\op{cyc}},E_2[p]).\]This is clear for $v\in \Sigma\backslash \Sigma_{\op{ss}}$. For $v\in \Sigma_{\op{ss}}$, this assertion follows from the arguments in \cite[p. 186]{BDKim}.
\end{proof} 

\begin{Prop}\label{isoresidualselmer}
Let $E$ be an elliptic curve over $F$ satisfying Hypothesis $\eqref{hypothesis}$ and $\ddag$ a signed vector. Let $\Sigma_0$ be a finite set of primes $v\nmid p$ containing $\Sigma_1(E)$. Then, there is an isomorphism
\[\op{Sel}^{\Sigma_0, \ddag}(E/F^{\op{cyc}})[p]\simeq \op{Sel}^{\Sigma_0, \ddag}(E[p]/F^{\op{cyc}}).\]
\end{Prop}
\begin{proof}
Let $\Sigma$ be a finite set of primes containing $\Sigma_0$, the primes at which $E$ has bad reduction and $\Sigma_p$. We consider the diagram relating the two Selmer groups: 
\[
\begin{tikzcd}[column sep = small, row sep = large]
0\arrow{r} & \op{Sel}^{\Sigma_0, \ddag}(E[p]/F^{\op{cyc}}) \arrow{r}  \arrow{d}{f} & H^1(\op{G}_{\Sigma}(F^{\op{cyc}}), E[p]) \arrow{r} \arrow{d}{g} & \op{im} \closure{\Phi}_{E}^{\Sigma_0,\ddag} \arrow{r} \arrow{d}{h} & 0\\
0\arrow{r} & \op{Sel}^{\Sigma_0, \ddag}(E/F^{\op{cyc}}) [p] \arrow{r} & H^1(\op{G}_{\Sigma}(F^{\op{cyc}}), E[p^{\infty}])[p]  \arrow{r}  &\op{im} \Phi_{E}^{\Sigma_0,\ddag}[p]\arrow{r}  & 0.\\
\end{tikzcd}\]
Since $\Gamma$ is pro-$p$ and $E[p]$ is an irreducible Galois module, clearly \[H^0(F,E[p])=H^0(F^{\op{cyc}},E[p])^{\Gamma}=0.\] Hence we deduce that $H^0(F^{\op{cyc}}, E[p^{\infty}])=0$ and have shown that $g$ is an isomorphism.
\par It only remains to show that $h$ is injective. For $v\in \Sigma/(\Sigma_{\op{ss}}(E)\cup \Sigma_0)$ denote by $h_v$ the natural map
\[h_v: \mathcal{H}_v(F^{\op{cyc}},E[p])\rightarrow \mathcal{H}_v(F^{\op{cyc}},E[p^{\infty}])\] and for $v\in \Sigma_{\op{ss}}(E)$,
\[h_v: \mathcal{H}_v^{\dagger_v}(F^{\op{cyc}},E[p])\rightarrow \mathcal{H}_v^{\dagger_v}(F^{\op{cyc}},E[p^{\infty}]).\]
We show that the maps $h_v$ are injective for $v\in \Sigma\backslash \Sigma_0$. This has been shown in the proof of \cite[Proposition 2.8]{greenbergvatsal} for $v\in \Sigma_p\backslash \Sigma_{\op{ss}}$ and in \cite[Proposition 2.10]{BDKim} for $v\in \Sigma_{\op{ss}}$. Therefore, it remains to consider primes $v\notin \Sigma\backslash (\Sigma_0\cup \Sigma_p)$. 
\par First consider the case when $p\geq 5$. Consider two further cases, first consider the case when $v\nmid(\mathcal{N}_E/\closure{\mathcal{N}}_E)$. In this case, the injectivity of $h_v$ follows from the proof of \cite[Lemma 4.1.2]{EPW}. Next, consider the case when $v|(\mathcal{N}_E/\closure{\mathcal{N}}_E)$. Since $v$ is not contained in $\Sigma_0$, it follows that $\mu_p$ is contained in $F_v$ and $E$ has either non-split multiplicative reduction or additive reduction at $v$. In this case, the injectivity of $h_v$ follows from \cite[Proposition 5.1]{hachimat}. When $p=3$, the injectivity of $h_v$ follows from the same reasoning as above.
\end{proof}
\begin{Prop}\label{nofinitelambdasubs}
Let $E$ be an elliptic curve over a number field $F$ and $\Sigma_0$ any finite set of primes $v\nmid p$. Assume that: $(\mathrm{i})$ $E(F)[p]=0$, $(\mathrm{ii})$ $\op{Sel}^{\Sigma_0, \ddag} (E/F^{\op{cyc}})$ is cotorsion as a $\Z_p[[\Gamma]]$-module. Then the Selmer group $\op{Sel}^{\Sigma_0, \ddag} (E/F^{\op{cyc}})$ contains no proper finite index $\Z_p[[\Gamma]]$-submodules.
\end{Prop}
\begin{proof} We adapt the proof of \cite[Proposition 4.14]{greenbergIWEC}, which is due to Greenberg. Let $\Sigma$ be a finite set of primes containing $\Sigma_0\cup \Sigma_p$ and the primes of $E$ at which $E$ has bad reduction. Consider the $\Sigma_0\cup \Sigma_p$-strict and relaxed Selmer groups:
\[\op{Sel}^{\op{rel}}(E/F^{\op{cyc}}):=\op{ker}\left\{H^1(F_{\Sigma}/F^{\op{cyc}},E[p^{\infty}])\rightarrow \prod_{v\in \Sigma\backslash (\Sigma_0\cup \Sigma_p)} \mathcal{H}_v(F^{\op{cyc}}, E[p^{\infty}])\right\},\]
\[\op{Sel}^{\op{str}}(E/F^{\op{cyc}}):=\op{ker}\left\{\op{Sel}^{\op{rel}}(E/F^{\op{cyc}})\rightarrow \prod_{v\in \Sigma_0\cup \Sigma_p} \prod_{\eta|v} H^1(F_{\eta}^{\op{cyc}}, E[p^{\infty}])\right\}.\]By Proposition $\ref{restrictionmapsurj}$, it follows that
\[\op{Sel}^{\op{rel}}(E/F^{\op{cyc}})/\op{Sel}^{\Sigma_0, \ddag} (E/F^{\op{cyc}})\simeq  \prod_{i=1}^d \mathcal{H}_{\mathfrak{p}_i}^{\ddag_i}(F^{\op{cyc}}, E[p^{\infty}]). \]For $i=1,\dots, d$, it follows from standard arguments (see the proof of \cite[Proposition 2.11]{BDKim}) that $ \mathcal{H}_{\mathfrak{p}_i}^{\ddag_i}(F^{\op{cyc}}, E[p^{\infty}])^{\vee}$ is isomorphic to $\Z_p[[\Gamma]]$. By \cite[Lemma 2.6]{greenbergvatsal}, it suffices to show that $\op{Sel}^{\op{rel}} (E/F^{\op{cyc}})$ has no proper finite index $\Z_p[[\Gamma]]$-submodules. Since $\op{Sel}^{\op{str}}(E/F^{\op{cyc}})^{\vee}$ is a quotient of $\op{Sel}^{\Sigma_0, \ddag} (E/F^{\op{cyc}})^{\vee}$, it is $\Z_p[[\Gamma]]$-torsion.
\par Recall that $\kappa$ denotes the $p$-adic cyclotomic character. For $s\in \Z$, let $A_s$ denote the twisted Galois module $E[p^{\infty}]\otimes \kappa^s $. Since $E(F)[p]=0$ and $\Gamma$ is pro-$p$, it follows that $E(F^{\op{cyc}})[p]=0$, and as a result, $H^0(F^{\op{cyc}}, A_s)=0$. Since $\Gamma$ is pro-$p$ it follows from standard arguments that $H^0(F, A_s)=0$ for all $s$. For a subfield $K$ of $F^{\op{cyc}}$, and $v$ a prime of $F$ which does not divide $p$, set $\mathcal{H}_v(K, A_s)$ to be the product $\prod_{\eta|v} H^1(K_{\eta}, A_s)/(A_s(K_{\eta})\otimes \Q_p/\Z_p)$, where $\eta$ ranges the finitely many primes of $K$ above $v$. Set $P^{\Sigma, \op{rel}}(K, A_s)$ to be the product \[P^{\Sigma, \op{rel}}(K, A_s):=\prod_{v\in \Sigma\backslash (\Sigma_0\cup \Sigma_p)} \mathcal{H}_v(K, A_s)\] and set $P^{\Sigma, \op{str}}(K, A_s)$ to be the product \[P^{\Sigma, \op{str}}(K, A_s):=\prod_{v\in \Sigma\backslash (\Sigma_0\cup \Sigma_p)} \mathcal{H}_v(K, A_s)\times \prod_{v\in \Sigma_0\cup \Sigma_p} H^1(K, A_s).\] Let $S_{A_s}^{\op{rel}}(K)$ and $S_{A_s}^{\op{str}}(K)$ be the Selmer groups defined as follows \[S_{A_s}^{\op{rel}}(K):=\ker\left(H^1(F_{\Sigma}/F, A_s)\rightarrow P^{\Sigma, \op{rel}}(K, A_s)\right),\] \[S_{A_s}^{\op{str}}(K):=\ker\left(H^1(F_{\Sigma}/F, A_s)\rightarrow P^{\Sigma, \op{str}}(K, A_s)\right).\] Since $\op{Sel}^{\op{str}}(E/F^{\op{cyc}})$ is $\Z_p[[\Gamma]]$-cotorsion, we have that $S_{A_s}^{\op{str}}(F^{\op{cyc}})^{\Gamma}$ is finite for all but finitely many values of $s$. Hence, $S_{A_s}^{\op{str}}(F)$ is finite for all but finitely many values of $s$. We set $M=A_s$, and in accordance with the proof of \cite[Proposition 4.14]{greenbergIWEC}, $M^*=A_{-s}$. Denote by $S_M(F)$ the Selmer group defined by the relaxed conditions $S_{A_s}^{\op{rel}}(F)$ and in accordance with the discussion on \cite[p. 100]{greenbergIWEC},  $S_{M^*}(F)$ is the strict Selmer group $S_{A_{-s}}^{\op{str}}(F)$. Let $s$ be such that $S_{M^*}(F)$ is finite. Since $S_{M^*}(F)$ is finite and $M^*(F)=0$, it follows that the map $H^1(F_{\Sigma}/F, M)\rightarrow P^{\Sigma, \op{rel}}(K, M)$ is surjective (see \cite[Proposition 4.13]{greenbergIWEC}). It follows from the proof of \cite[Proposition 4.14]{greenbergIWEC} that $\op{Sel}^{\op{rel}} (E/F^{\op{cyc}})$ has no proper finite index $\Z_p[[\Gamma]]$-submodules. This completes the proof.
\end{proof}
Recall that for $i=1,2$, the $\mu$-invariant (resp. $\lambda$-invariant) of the Selmer group $\op{Sel}^{\Sigma_1,\ddag}(E_i/F^{\op{cyc}})$ is denoted $\mu_{E_i}^{\Sigma_1,\ddag}$ (resp. $\lambda_{E_i}^{\Sigma_1,\ddag}$).
\begin{Th}\label{equalitymulambda}
Let $E_1$ and $E_2$ be elliptic curves over $F$ which are $p$-congruent. Let $\ddag$ be a signed vector and assume that $(E_1,E_2,\ddag)$ satisfies Hypothesis $\eqref{hypothesis2}$. Let $\Sigma_1$ be the set of primes as in Definition $\ref{sigmazerodef}$.
Then the following assertions hold:
\begin{enumerate}
    \item The $\mu$-invariant $\mu_{E_1}^{\Sigma_1,\ddag}$ is equal to zero if and only if $\mu_{E_2}^{\Sigma_1,\ddag}$ is equal to zero.
    \item If $\mu_{E_1}^{\Sigma_1,\ddag}=0$ (or equivalently $\mu_{E_2}^{\Sigma_1,\ddag}=0$), then the $\Sigma_1$-imprimitive $\lambda$-invariants $\lambda_{E_1}^{\Sigma_1,\ddag}$ and $\lambda_{E_2}^{\Sigma_1, \ddag}$ are equal.
\end{enumerate}
\end{Th}
\begin{proof}
For $i=1,2$, let $M_i$ denote the Pontryagin dual of $\op{Sel}^{\Sigma_1,\ddag}(E_i/F^{\op{cyc}})$. It follows from Propositions $\ref{prop42}$ and $\ref{isoresidualselmer}$ that $M_1/pM_1$ is isomorphic to $M_2/pM_2$. Note that $\mu_{E_i}^{\Sigma_1,\ddag}$ is equal to $0$ if and only if $M_i/pM_i$ is finite. Thus, it follows that if $\mu_{E_1}^{\Sigma_1,\ddag}$ is zero, then so is $\mu_{E_2}^{\Sigma_1,\ddag}$. 
\par Next, assume that $\mu_{E_1}^{\ddag}$ and $\mu_{E_2}^{\ddag}$ are both zero. Proposition $\ref{nofinitelambdasubs}$ asserts that $M_i$ contains no finite $\Z_p[[\Gamma]]$-submodules, and thus, $M_i$ is a free $\Z_p$-module of rank equal to $\lambda_{E_i}^{\Sigma_1, \ddag}$. Since $M_1/pM_1$ is isomorphic to $M_2/pM_2$, it follows that $\lambda_{E_1}^{\Sigma_1, \ddag}$ is equal to $\lambda_{E_2}^{\Sigma_1, \ddag}$.
\end{proof}

\section{Congruences for Euler Characteristics}
\par Consider elliptic curves $E_1$ and $E_2$ that are $p$-congruent and let $\Sigma^{\op{ss}}=\{\mathfrak{p}_1,\dots, \mathfrak{p}_d\}$ be the set of supersingular primes $v$ of $E_1$ such that $v|p$. As noted earlier, these are also the supersingular primes $v$ of $E_2$ such that $v|p$. Let $\ddag\in \{+,-\}^d$ be a signed vector and assume that Hypothesis $\eqref{hypothesis2}$ is satisfied for the triple $(E_1, E_2, \ddag)$. Associated to $E_1$ and $E_2$ is the set of primes $\Sigma_1$, see Definition $\ref{sigmazerodef}$. In this section, it is shown that there is an explicit relationship between the multi-signed Euler characteristics $\chi_t^{\ddag}(\Gamma, E_1)$ and $\chi_t^{\ddag}(\Gamma, E_2)$. 
\par Let $E$ be an elliptic curve satisfying Hypothesis $\eqref{hypothesis}$ and $\Sigma_0$ a finite set of primes $v\nmid p$. Recall that $r_{E}^{\ddag}$ is the order of vanishing of the characteristic polynomial $f_{E}^{\ddag}(T)$ at $T=0$. The following Proposition shows that the quantity $r_{E}^{\ddag}$ is related to the Mordell Weil rank of $E$ and the reduction type of $E$ at the primes $v|p$.
\begin{Prop}
Let $E$ be as above and assume that the following two conditions are satisfied:
\begin{enumerate}[label=(\roman*)]
    \item the truncated Euler characteristic $\chi_t^{\ddag}(\Gamma, E)$ is defined.
    \item The $p$-primary part of the Tate-Shafarevich group $\Sh(E/F)$ is finite.
\end{enumerate} Let $r$ be the Mordell-Weil rank of $E$ and $\op{sp}_E$ the number of primes $v|p$ at which $E$ has split multiplicative reduction. Then, we have that $r\leq r_E^{\ddag}\leq r+\op{sp}_E$. In particular, if there are no primes $v\in \Sigma_p$ at which $E$ has split multiplicative reduction, then $r_E^{\ddag}$ is equal to $r$.

\end{Prop}
\begin{proof}
Let $\Sigma$ be a finite set of primes containing $\Sigma_p$ and the primes at which $E$ has bad reduction. By Lemma $\ref{TECbasiclemma}$, the multi-signed rank $r_E^{\ddag}$ is equal to the $\Z_p$-corank of $\op{Sel}^{\ddag}(E/F^{\op{cyc}})^{\Gamma}$. We compare $\op{Sel}^{\ddag}(E/F^{\op{cyc}})^{\Gamma}$ with the usual Selmer group $\op{Sel}(E/F)$. Let ${\Phi}_{E,F}: H^1(F_{\Sigma}/F, E[p^{\infty}])\rightarrow \prod_{v\in \Sigma} \mathcal{H}_v(F, E[p^{\infty}])$ be the defining map for the Selmer group $\op{Sel}(E/F)$. Consider the fundamental diagram
\[
\begin{tikzcd}[column sep = small, row sep = large]
0\arrow{r} & \op{Sel}(E/F) \arrow{r}  \arrow{d}{f} & H^1(F_{\Sigma}/F, E[p^{\infty}]) \arrow{r} \arrow{d}{g} & \op{im} {\Phi}_{E,F} \arrow{r} \arrow{d}{h} & 0\\
0\arrow{r} & \op{Sel}^{ \ddag}(E/F^{\op{cyc}})^{\Gamma} \arrow{r} & H^1(F_{\Sigma}/F^{\op{cyc}}, E[p^{\infty}])^{\Gamma} \arrow{r}  & \op{im} {\Phi}_{E,F^{\op{cyc}}}^{\ddag}\arrow{r}  & 0,\\
\end{tikzcd}\]
where the vertical maps are induced by restriction.
Since $E(F)[p]=0$ and $\Gamma$ is pro-$p$, it follows that $H^0(F^{\op{cyc}}, E[p^{\infty}])=0$. It follows from the inflation-restriction sequence that $g$ is an isomorphism. Thus, we have the following short exact sequence:
\begin{equation}\label{ses2}0\rightarrow \op{Sel}(E/F)\rightarrow \op{Sel}^{ \ddag}(E/F^{\op{cyc}})^{\Gamma}\rightarrow \op{ker}h\rightarrow 0.\end{equation}
It suffices to show that the kernel of $h$ has corank less than or equal to $\op{sp}_E$. For $v\in \Sigma\backslash \Sigma_{\op{ss}}$, let $h_v'$ denote the natural map 
\[h_v': \mathcal{H}_v(F, E[p^{\infty}])\rightarrow \mathcal{H}_v(F^{\op{cyc}}, E[p^{\infty}]).\] For $\mathfrak{p}_i\in \Sigma_{\op{ss}}$, set $h_{\mathfrak{p}_i}'$ denotes the map
\[h_{\mathfrak{p}_i}': \mathcal{H}_{\mathfrak{p}_i}^{\ddag_i}(F, E[p^{\infty}])\rightarrow \mathcal{H}_{\mathfrak{p}_i}^{\ddag_i}(F^{\op{cyc}}, E[p^{\infty}]).\]
\par Let $h'$ be the product of the maps $h_v'$ as $v$ ranges over $\Sigma$. Note that the kernel of $h$ is contained in the kernel of $h'$. Let $\Sigma_{\op{sp}}(E)$ be the set of primes $v|p$ at which $E$ has split multiplicative reduction. We show that $\op{ker} h_v'$ is finite if $v\notin \Sigma_{\op{sp}}(E)$ and of corank one for $v\in \Sigma_{\op{sp}}(E)$. For $v\in \Sigma\backslash \Sigma_p$, this follows from \cite[Lemma3.4]{GCEC}, for $v\in \Sigma_p$ at which $E$ has good ordinary reduction, it follows from \cite[Proposition 3.5]{GCEC}. Next consider the case where $v|p$ and $E$ has multiplicative reduction at $v$.
\par The Galois representation of the local Galois group $\op{G}_{F_v}$ on $E[p]$ is of the form $\mtx{\varphi_v \kappa}{\ast}{}{\varphi_v^{-1}}$, where $\varphi_v$ is an unramified character. Let $w$ be a prime of $F^{\op{cyc}}$ above $v$ and let $\Gamma_v$ denote $\op{Gal}(F_w^{\op{cyc}}/F_v)$. By Shapiro's Lemma,
\[H^0(\Gamma, \prod_{\eta|v} H^1(F_{\eta}^{\op{cyc}}, D_v))\simeq H^0(\Gamma_v, H^1(F_{w}^{\op{cyc}}, D_v)).\] The kernel of $h_v'$ is contained in the kernel of the restriction map 
\[H^1(F_v, D_v)\rightarrow H^1(F_w^{\op{cyc}}, D_v)\] and by the inflation restriction sequence, the kernel of $h_v'$ is contained in $H^1(\Gamma_v, H^0(F_w^{\op{cyc}},D_v))$. When $E$ has nonsplit multiplicative reduction at $v$, the character $\varphi_v$ is non-trivial and thus $H^0(F_v, D_v)=0$. Since $\Gamma_v$ is pro-$p$, it follows that $H^0(F_w^{\op{cyc}},D_v)=0$ when $\varphi_v\neq 1$. On the other hand, when $\varphi_v=1$, the Galois action on $H^0(F_{w}^{\op{cyc}}, D_v)\simeq \Q_p/\Z_p$ is trivial. As a result, the corank of $H^1(\Gamma_v, H^0(F_{w}^{\op{cyc}}, D_v))$ is one when $E$ has split multiplicative reduction at $v$.
\par Finally, consider the case when $v|p$ and $E$ has supersingular reduction at $v$. In this case, $h_v'$ is injective, as shown in the proof of \cite[Theorem 5.3]{leilim2}.
\par Putting all this together, we have that $\op{corank}_{\Z_p} \op{ker} h\leq \op{sp}_E$, and thus, from the short exact sequence $\eqref{ses2}$, we have
\[r_E^{\ddag}\leq \op{corank}_{\Z_p}\op{Sel}(E/F)+\op{sp}_E=r+ \op{corank}_{\Z_p}\Sh(E/F)[p^{\infty}]+\op{sp}_E.\]Since it is assumed that $\Sh(E/F)[p^{\infty}]$ is finite, the result follows.
\end{proof}

\par We now study congruences for truncated Euler characteristics. For $v\in \Sigma_0$, set $h_{E}^{(v)}(T)$ to be the characteristic polynomial of the Pontryagin dual of $\mathcal{H}_v(F^{\op{cyc}},E[p^{\infty}])$. Since $ \op{Sel}^{\ddag}(E/F^{\op{cyc}})$ is $\Z_p[[\Gamma]]$-cotorsion, it follows from Proposition $\ref{restrictionmapsurj}$ that there is a short exact sequence:
\[0\rightarrow \Sel^{\ddag}(E/F^{\op{cyc}})\rightarrow \Sel^{\ddag,\Sigma_0}(E/F^{\op{cyc}})\rightarrow \prod_{v\in \Sigma_0} \mathcal{H}_v(F^{\op{cyc}}, E[p^{\infty}])\rightarrow 0.\] As a result, we arrive at the following relation:
\begin{equation}\label{imprimitivetoprimitive}f_{E}^{\Sigma_0, \ddag}(T)=f_{E}^{\ddag}(T)\prod_{l\in \Sigma_0} h_{E}^{(v)}(T).\end{equation}
Recall that $V_p(E)$ is the $p$-adic vector space $T_p(E)\otimes \Q_p$, equipped with $\op{Gal}(\closure{F}/F)$-action. Let $P_v(E,T)$ denote the characteristic polynomial \begin{equation}\label{pvdefinition}P_v(E,T):=\op{det}\left((\op{Id}-\op{Frob}_v X)_{|V_p(E)^{\op{I}_v}}\right),\end{equation}where $\op{I}_v$ is the inertia group at $v$. For $s\in \mathbb{C}$, set $L_v(E,s)$ to denote $P_v(E,Nv^{-s})^{-1}$, where $Nv$ denotes the norm $N_{L/\Q}v$. Recall that  $\gamma$ is a topological generator of $\Gamma$. Let $\rho:\Gamma\rightarrow \mu_{p^{\infty}}$ be a finite order character and let $\sigma_v$ denote the Frobenius at $v$. Let $\mathcal{P}_v(E,T)$ be the element in $\Z_p[[\Gamma]]$ defined by the relation
$\mathcal{P}_v(E,\rho(\gamma)-1)=P_v(E,\rho(\sigma_v) Nv^{-1})$. According to \cite[Proposition 2.4]{greenbergvatsal}, the polynomial $\mathcal{P}_v(E,T)$ generates the Pontryagin dual of $\mathcal{H}_v(F^{\op{cyc}}, E[p^{\infty}])$ and thus coincides with $h^{(v)}(T)$ up to a unit in $\Z_p[[\Gamma]]$. 
\begin{Def}\label{PhiDef} Set $\Phi_{E,\Sigma_0}$ to be the product $\prod_{v\in \Sigma_0} |L_v(E, 1)|_p$. Here, $|L_v(E, 1)|_p$ is set to be equal to $0$ if $L_v(E, 1)^{-1}=0$.
\end{Def}
Write \[f_{E}^{\Sigma_0,\ddag}(T)= T^{r_E^{\ddag}} g_{E}^{\Sigma_0,\ddag}(T),\] and set $|g_{E,\Sigma_0}(0)|_{p}^{-1}$ to be zero if $g_{E,\Sigma_0}(0)$ equals zero. It follows from the relation $\eqref{imprimitivetoprimitive}$ that \begin{equation}\label{imprimitivetoprimitive2}g_{E}^{\Sigma_0, \ddag}(T)=g_{E}^{\ddag}(T)\prod_{l\in \Sigma_0} h_{E}^{(v)}(T).\end{equation}One has the following result.
\begin{Lemma}\label{Lemma53}
Let $E$ be an elliptic curve satisfying Hypothesis $\eqref{hypothesis}$ and let $\ddag$ be a signed vector, for a finite set $\Sigma_0$ consisting of primes $v\nmid p$. Then we have the following equality:
\[\Phi_{E,\Sigma_0}\times \chi_t(\Gamma, E)=|g_{E}^{\Sigma_0, \ddag}(0)|_{p}^{-1}.\]
\end{Lemma}
\begin{proof}
Note that $\mathcal{P}_v(E,0)$ equals $P_v(E,l^{-1})=L_v(E,1)^{-1}$. Therefore, it follows that $h^{(v)}(0)$ coincides with $|L_v(E,1)|_p$ up to a unit in $\Z_p$. From the relation $\eqref{imprimitivetoprimitive2}$, we have that \[|g_{E}^{\Sigma_0, \ddag}(0)|_p^{-1}= |g_{E}^{\ddag}(0)|_p^{-1}\prod_{l\in \Sigma_0} |L_v(E,1)|_p=|g_{E}^{\ddag}(0)|_p^{-1}\times \Phi_{E,\Sigma_0}.\] Lemma $\ref{TECbasiclemma}$ asserts that $\chi_t(\Gamma, E)$ is equal to $|g_{E}^{\ddag}(0)|_p^{-1}$, and this completes the proof.
\end{proof}
Denote the $\mu$-invariants of $\op{Sel}^{\ddag}(E/F^{\op{cyc}})$ and $\op{Sel}^{\Sigma_0,\ddag}(E/F^{\op{cyc}})$ by $\mu_{E}^{ \ddag}$ and $\mu_{E}^{\Sigma_0, \ddag}$ respectively. Denote the $\lambda$-invariants by $\lambda_{E}^{ \ddag}$ and $\lambda_{E}^{\Sigma_0, \ddag}$ respectively. It is easy to show that for $v\in \Sigma_0$, the $\mu$-invariant of $h_{E}^{(v)}(T)$ is equal to zero, and hence $\mu_{E}^{\Sigma_0, \ddag}$ is equal to $\mu_E^{\ddag}$.
\begin{Lemma}\label{pretheoremlemma}
Let $E$ be an elliptic curve satisfying Hypothesis $\eqref{hypothesis}$ and let $\ddag$ be a signed vector. For a finite set $\Sigma_0$ consisting of primes $v\nmid p$, we have that \[\Phi_{E,\Sigma_0}\times \chi_t(\Gamma, E)=1\Leftrightarrow \mu_E^{\ddag}=0 \text{ and }\lambda_{E}^{\Sigma_0, \ddag}=r_E^{\ddag}.\]
\end{Lemma}
\begin{proof}
First assume that $\Phi_{E,\Sigma_0}\times \chi_t(\Gamma, E)=1$, it is then shown that $\mu_E^{\ddag}=0$ and $\lambda_{E}^{\Sigma_0, \ddag}=r_E^{\ddag}$. Recall that $f_{E}^{\Sigma_0,\ddag}(T)$ is the characteristic polynomial for the Selmer group $\op{Sel}^{\Sigma_0,\ddag}(E/F^{\op{cyc}})$ and is written as $T^{r_E^{\ddag}}g_{E}^{\Sigma_0,\ddag}(T)$. It follows from Lemma $\ref{Lemma53}$ that $g_{E}^{\Sigma_0,\ddag}(0)$ is a unit in $\Z_p$, and thus $g_{E}^{\Sigma_0,\ddag}(T)$ is a unit in $\Z_p[[\Gamma]]$. It follows from this that $\mu_{E,\Sigma_0}^{\ddag}=0$ and $\lambda_{E,\Sigma_0}^{\ddag}=r_{E}^{\ddag}$. The following relation is satisfied:
\[f_{E,\Sigma_0}^{\ddag}(T)=f_E^{\ddag}(T)\times \prod_{v\in \Sigma_0} h^{(v)}_E(T)\](see $\eqref{imprimitivetoprimitive}$). According to \cite[Proposition 2.4]{greenbergvatsal}, the polynomial $\mathcal{P}_v(E,T)$ generates the Pontryagin dual of $\mathcal{H}_v(F^{\op{cyc}}, E[p^{\infty}])$ and thus coincides with $h^{(v)}(T)$ up to a unit in $\Z_p[[\Gamma]]$. It is easy to see that $\mathcal{P}_v(E,T)$ is equal to a monic polynomial, up to a unit in $\Z_p$. As a result, $p$ does not divide $h^{(v)}(T)$. It follows that $\mu_{E}^{\ddag}=\mu_{E,\Sigma_0}^{\ddag}=0$. 
\par Conversely, suppose that $\mu_{E}^{\ddag}=0$ and that $\lambda_{E,\Sigma_0}^{\ddag}=r_E^{\ddag}$. The degree of $f_{E,\Sigma_0}^{\ddag}(T)$ is equal to $\lambda_{E,\Sigma_0}^{\ddag}=r_E^{\ddag}$. Since $f_{E,\Sigma_0}^{\ddag}(T)$ is expressed as $T^{r_E^{\ddag}} g_{E,\Sigma_0}(T)$, the degree of $g_{E,\Sigma_0}(T)$ is equal to zero. It has been shown that $\mu_{E,\Sigma_0}^{\ddag}=\mu_{E}^{\ddag}=0$. It follows that $g_{E,\Sigma_0}(T)$ is a unit in $\Z_p$. Lemma $\ref{Lemma53}$ asserts that $\Phi_{E,\Sigma_0}\times \chi_t(\Gamma, E)$ is equal to $|g_{E,\Sigma_0}(0)|_p^{-1}$ and therefore equal to $1$.
\end{proof}
Next, we come to the main theorem of the section. It is shown that the truncated Euler characteristic of $E_1$ is related to that of $E_2$ after one accounts for the auxiliary set of primes $\Sigma_1$. The smaller the set of primes $\Sigma_1$, the more refined the congruence relation between truncated Euler characteristics is. This is why it is of considerable importance that the set of primes $\Sigma_1$ be carefully chosen to be as small as possible. We note that in the statement of \cite[Theorem 3.3]{raysujatha}, the elliptic curves $E_1$ and $E_2$ were defined over $\Q$ with good ordinary reduction at $p$. Furthermore, the set of auxiliary primes $\Sigma_0$ was the full set of primes $v\nmid p$ at which $E_1$ or $E_2$ has bad reduction. The set of primes $\Sigma_1$ is smaller and hence, when $F=\Q$ the result below refines the result \cite[Theorem 3.3]{raysujatha}.
\begin{Th}\label{gammacongruence}
Let $E_1$ and $E_2$ be elliptic curves which satisfy Hypothesis $\eqref{hypothesis2}$ mentioned in the introduction. Then, the following assertions hold.
\begin{enumerate}
    \item\label{gammacongruencec1} Suppose $r_{E_1}^{\ddag}$ is equal to $r_{E_2}^{\ddag}$. Then, $\Phi_{E_1,\Sigma_1}\times  \chi_t(\Gamma, E_1)$ is equal to $1$ if and only if $\Phi_{E_2,\Sigma_1}\times \chi_t(\Gamma,E_2)$ is equal to $1$.
    \item\label{gammacongruencec2} If $r_{E_1}^{\ddag}<r_{E_2}^{\ddag}$, then $\Phi_{E_1,\Sigma_1}\times \chi_t(\Gamma,E_1)$ is divisible by $p$.
\end{enumerate}
\end{Th}

\begin{proof}
\par We first assume consider the case when $r_{E_1}^{\ddag}$ is equal to $r_{E_2}^{\ddag}$.
It follows from Lemma $\ref{pretheoremlemma}$ that $\Phi_{E_i,\Sigma_1}\times  \chi_t(\Gamma, E_i)=1$ if and only if $\mu_{E_i}^{\ddag}=0$ and $\lambda_{E_i}^{\Sigma_1,\ddag}=r_{E_i}^{\ddag}$. Suppose that $\Phi_{E_1,\Sigma_1}\times  \chi_t(\Gamma, E_1)=1$, then it follows that $\mu_{E_1}^{\ddag}=0$ and $\lambda_{E_1}^{\Sigma_1,\ddag}=r_{E_1}^{\ddag}$. It follows from Corollary $\ref{equalitymulambda}$ that $\mu_{E_2}^{\ddag}=0$ and $\lambda_{E_2}^{\Sigma_1,\ddag}=r_{E_2}^{\ddag}$. Thus, by Lemma $\ref{pretheoremlemma}$, we have that $\Phi_{E_2,\Sigma_1}\times  \chi_t(\Gamma, E_2)=1$. This proves part $\eqref{gammacongruencec1}$.
\par Next, it is assumed that $r_{E_1}^{\ddag}<r_{E_2}^{\ddag}$ and it shown that $p$ divides $\Phi_{E_1,\Sigma_1}\times \chi_t(\Gamma,E_1)$. Suppose by way of contradiction that $p$ does not divide $\Phi_{E_1,\Sigma_1}\times \chi_t(\Gamma,E_1)$. This means that $\Phi_{E_1,\Sigma_1}\times \chi_t(\Gamma,E_1)$ is equal to $1$. Lemma $\ref{pretheoremlemma}$ asserts that $\mu_{E_1}^{\ddag}=0$ and $\lambda_{E_1}^{\Sigma_1,\ddag}=r_{E_1}^{\ddag}$. It follows from Corollary $\ref{equalitymulambda}$ that $\mu_{E_2}^{\ddag}=0$ and $\lambda_{E_2}^{\Sigma_1,\ddag}=r_{E_1}^{\ddag}$. Recall that $\lambda_{E_2}^{\Sigma_1,\ddag}$ is the degree of the characteristic polynomial $f_{E_2}^{\ddag}(T)$. Since $r_{E_2}^{\ddag}$ is the order of vanishing of $f_{E_2}^{\ddag}(T)$ at $T=0$, it follows that $\lambda_{E_2}^{\Sigma_1,\ddag}\geq r_{E_2}^{\ddag}$. This is a contradiction, as it is assumed that $r_{E_1}^{\ddag}<r_{E_2}^{\ddag}$. Hence, $p$ divides $\Phi_{E_1,\Sigma_1}\times \chi_t(\Gamma,E_1)$. This proves part $\eqref{gammacongruencec2}$.
\end{proof}
\par Consider the case when $r_{E_1}^{\ddag}=r_{E_2}^{\ddag}$. It is natural to ask if $\chi_t(\Gamma,E_1)=1$ implies that $\chi_i(\Gamma, E_2)=1$? This statement is false, as is shown in \cite[Example 5.2]{raysujatha}, where it is shown that there are $5$-congruent elliptic curves $E_1$ and $E_2$ over $\Q$ which both have Mordell-Weil rank $1$ and good ordinary reduction at $5$, such that $\chi_t(\Gamma, E_1)=1$ and $\chi_t(\Gamma, E_2)=5^2$. Therefore it is necessary to account the factors $\Phi_{E_i, \Sigma_1}$, i.e. \[\Phi_{E_1, \Sigma_1}\times \chi_t(\Gamma, E_1)=1\Leftrightarrow\Phi_{E_2, \Sigma_1}\times \chi_t(\Gamma, E_2)=1.\]
\section{Results over $\Z_p^m$-extensions}
\par We show that our results on truncated Euler characteristics generalize to $\Z_p^m$-extensions. The multi-signed Selmer groups have not been defined for more general $p$-adic Lie extensions. In the good ordinary reduction setting, the reader is referred to \cite{Zerbes} for a discussion on truncated Euler characteristics over general $p$-adic Lie extensions. In \cite[ section 4]{raysujatha}, congruence relations are proved for truncated Euler characteristics over $p$-adic Lie extensions for elliptic curves with good ordinary reduction at $p$.
\par Let $\mathcal{F}_{\infty}/F$ be a Galois extension of $F$ containing $F^{\op{cyc}}$ such that $\op{Gal}(\mathcal{F}_{\infty}/F)\simeq \Z_p^{m}$ for an integer $m\geq 1$. Set $\op{G}:=\op{Gal}(\mathcal{F}_{\infty}/F)$, $H:=\op{Gal}(\mathcal{F}_{\infty}/F^{\op{cyc}})$ and identify $\Gamma$ with $\op{Gal}(F^{\op{cyc}}/F)\simeq \op{G}/\op{H}$. Let $E$ be an elliptic curve and $\ddag$ be a signed vector, assume that $(E,\ddag)$ satisfies Hypothesis $\eqref{hypothesis1}$ and assume that Hypothesis $\eqref{Finfhyp}$ is satisfied for the pair $(E, \mathcal{F}_{\infty})$. Another natural hypothesis is that the Selmer groups be in $\mathfrak{M}_H(\op{G})$. The precise definition is given below.
\begin{Def}
Let $M$ be a cofinitely generated cotorsion $\Z_p[[\op{G}]]$-module and let $M(p)$ be the $p$-primary submodule of $M$. The category $\mathfrak{M}_H(\op{G})$ consists of all such modules $M$ such that $M/M(p)$ is a finitely generated $\Z_p[[H]]$-module.
\end{Def} We introduce the following hypothesis.

\begin{hyp}\label{MHG}
Assume that the Selmer group $\op{Sel}^{\ddag}(E/\mathcal{F}_{\infty})$ is in $\mathfrak{M}_H(\op{G})$.
\end{hyp}
We recall the notion of the truncated $\op{G}$-Euler characteristic of $\op{Sel}^{\ddag}(E/\mathcal{F}_{\infty})$. 
\begin{Def}Let $M$ be a discrete $p$-primary $\op{G}$-module, define 
\[\psi_D: H^0(\Gamma,M^H)\rightarrow H^1(\op{G},M)\] as the composite of the natural map $\phi_{D^H}:H^0(\Gamma, M^H)\rightarrow H^1(\Gamma, M^H)$ with the inflation map $H^1(\Gamma, M^H)\rightarrow H^1(\op{G},M)$. The module $M$ has finite truncated $\op{G}$-Euler characteristic if the following two conditions are satisfied:
\begin{enumerate}
    \item both $\operatorname{ker}(\psi_M)$ and $\operatorname{cok}(\psi_M)$ are finite,
    \item $H^i(\op{G}, M)$ is finite for $i\geq 2$.
\end{enumerate} The truncated $\op{G}$-Euler characteristic is then defined by
\[\chi_t(\op{G},M):=\frac{\# \operatorname{ker}(\psi_M)}{\# \operatorname{cok}(\psi_M)}\times \prod_{i\geq 2}^d \# \left(H^i(\op{G},M)\right)^{(-1)^i}.\]
\end{Def}
Next, we recall the notion of the Akashi series. For a cofinitely generated $\Gamma$-module $M$, let $\op{char}_{\Gamma}(M)$ be the characteristic polynomial of its Pontryagin dual $M^{\vee}$. This is the unique polynomial generating the characteristic ideal which can be expressed as the product of a power of $p$ and a distinguished polynomial.
\begin{Def}Let $D$ be a discrete $p$-primary $\op{G}$-module. The Akashi series of $D$ is defined if $H^i(H, D)$ is cotorsion and cofinitely generated for all $i\geq 0$. In this case, the Akashi series $\op{Ak}_H(D)$ is taken to be the following alternating product:
\[\op{Ak}_H(D):=\prod_{i\geq 0}^{d-1} \left(\op{char}_{\Gamma} H^i(H, D)\right)^{(-1)^i}.\]
\end{Def}
Note that the Akashi series coincides with the characteristic polynomial when $\mathcal{F}_{\infty}=F^{\op{cyc}}$. The following Proposition describes the relationship between the Akashi series and truncated $\op{G}$-Euler characteristic.
\begin{Prop}\label{prop63}\cite[Proposition 2.10]{Zerbes}
Let $D$ be a discrete $p$-primary $\op{G}$-module such that the Pontryagin dual $D^{\vee}$ is in $\mathfrak{M}_H(\op{G})$. Then, the Akashi series $\op{Ak}_H(D)$ is defined. Furthermore, suppose that
$D$ has finite truncated $\op{G}$-Euler characteristic. Let $\beta T^k$ be the leading term of $\op{Ak}_H(D)$. Then, the following assertions hold:
\begin{enumerate}
    \item the order of vanishing $k$
is given by
\[k = \sum_{i\geq 0}
(-1)^i \op{corank}_{\Z_p} H^i(H,D)^{\Gamma}\]

\item the truncated $\op{G}$-Euler characteristic is given by \[\chi_t(\op{G},D) = |\beta|_p^{-1}.\]
\end{enumerate}
\end{Prop}The truncated $\op{G}$-Euler characteristic of $\op{Sel}^{\ddag}(E/\mathcal{F}_{\infty})$ is denoted $\chi_t^{\ddag}(\op{G}, E)$. When the truncated $\op{G}$-Euler characteristic is defined, it is related to the Akashi series of $\op{Sel}^{\ddag}(E/\mathcal{F}_{\infty})$, which we denote by $\op{Ak}_H^{\ddag}(E)$. The following Theorem due to Lei and Lim describes the link between Akashi series $\op{Ak}_H^{\ddag}(E)$ and the characteristic polynomial $f_{E}^{\ddag}(T)$ of the Selmer group $\op{Sel}^{\ddag}(E/F^{\op{cyc}})$.
\begin{Th}\cite[Theorem 1.1]{leilim2}\label{leilimthm} Let $E$ be an elliptic curve over a number field $F$, and let $\ddag$ be a signed vector. Let $\mathcal{F}_{\infty}$ be a $\Z_p^m$-extension of $F$ which contains $F^{\op{cyc}}$. Assume that:
\begin{enumerate}
    \item Hypotheses $\eqref{hypothesis}, \eqref{Finfhyp}$ and $\eqref{MHG}$ are satisfied.
    \item The Selmer group $\op{Sel}^{\ddag}(E/F^{\op{cyc}})$ is $\Z_p[[\Gamma]]$-cotorsion.
\end{enumerate} Then, the Akashi
series $\op{Ak}_H^{\ddag}(E)$ is well-defined and is given by
\[\op{Ak}_H^{\ddag}(E)=T^{\gamma_E}\cdot f_E^{\ddag}(T),
\]
where $\gamma_E$ is the number of primes of $F^{\op{cyc}}$ above $p$ with nontrivial decomposition group in $\mathcal{F}_{\infty}/F^{\op{cyc}}$ and
at which $E$ has split multiplicative reduction.
\end{Th}

\begin{Cor}\label{lastcor}
 Let $E$ be an elliptic curve over a number field $F$, and let $\ddag$ be a signed vector. Let $\mathcal{F}_{\infty}$ be a $\Z_p^m$-extension of $F$ which contains $F^{\op{cyc}}$. Assume that
 \begin{enumerate}
     \item Hypotheses $\eqref{hypothesis}$, $ \eqref{Finfhyp}$, $\eqref{hypothesis1}$ and $\eqref{MHG}$ are satisfied.
     \item The truncated $\op{G}$-Euler characteristic $\chi_t^{\ddag}(\op{G}, E)$ is well defined.
 \end{enumerate} Then, the truncated $\op{G}$-Euler characteristic $\chi_t(\op{G}, E)$ is equal to the truncated $\Gamma$-Euler characteristic $\chi_t(\Gamma, E)$.
\end{Cor}
\begin{proof}
The result is a direct consequence of Theorem $\ref{leilimthm}$, Proposition $\ref{prop63}$ and Lemma $\ref{TECbasiclemma}$.
\end{proof}
\begin{Th}\label{Gcongruence}
Let $E_1$ and $E_2$ be $p$-congruent elliptic curves over a number field $F$. Let $\ddag$ be a signed vector. Assume that
\begin{enumerate}
     \item Hypotheses $ \eqref{Finfhyp}$, $\eqref{hypothesis2}$ and $\eqref{MHG}$ are satisfied.
     \item The truncated $\op{G}$-Euler characteristic $\chi_t^{\ddag}(\op{G}, E)$ is well defined.
 \end{enumerate}
Then, the following assertions hold.
\begin{enumerate}
    \item Suppose that $r^{\ddag}_{E_1}=r^{\ddag}_{E_2}$. Then, we have that the truncated $\op{G}$-Euler characteristics are defined, then $\Phi_{E_1,\Sigma_1}\times \chi_t(\op{G},E_1)$ is equal to $1$ if and only if $\Phi_{E_2,\Sigma_1}\times \chi_t(\op{G}, E_2)$ is equal to $1$. Furthermore, $\Phi_{E_1,\Sigma_1}\times \op{Ak}_H(E_1)$ is equal to $T^{r_{E_1}+\gamma_{E_1}}$ if and only if $\Phi_{E_2,\Sigma_1}\times \op{Ak}_H(E_2)$ is equal to $T^{r_{E_1}+\gamma_{E_2}}$.
    \item Suppose that $r^{\ddag}_{E_1}<r^{\ddag}_{E_2}$. Then, $p$ divides $\Phi_{E_1,\Sigma_1}\times \chi_t(\op{G},E_1)$.
\end{enumerate}
\end{Th}
\begin{proof}The result is a direct consequence of Theorem $\ref{gammacongruence}$, Proposition $\ref{prop63}$ and Corollary $\ref{lastcor}$. 
\end{proof}
\section{Examples}\label{examples}
\par We discuss two concrete examples which illustrate our results for $p=5$. Our computations are aided by Sage. The reader is referred to \cite[section 5]{raysujatha} for examples when both elliptic curves are defined over $\Q$ and have good ordinary reduction at $p=5$.
\subsection{Example 1} Consider elliptic curves $E_1=66a1$ and $E_2=462d1$. The elliptic curves $E_1$ and $E_2$ are $5$-congruent and supersingular at $5$. Both elliptic curves $E_1$ and $E_2$ have Mordell-Weil rank zero, hence the Euler characteristic $\chi(\Gamma, E_i)$ is well defined for $i=1,2$. Hypothesis $\eqref{hypothesis}$ is satisfied for the elliptic curves $E_1$ and $E_2$. Let $\ddag\in \{+,-\}$ be a choice of sign. Assume that for $i=1,2$, the Selmer group $\op{Sel}^{\ddag}(E_i/\Q^{\op{cyc}})$ is $\Z_p[[\Gamma]]$-cotorsion. Therefore, Hypothesis $\eqref{hypothesis2}$ is satisfied. We show that $\Phi^{\ddag}_{E_i, \Sigma_1}=1$ and $\chi_t^{\ddag}(\Gamma, E_i)=1$ for $i=1,2$. This demonstrates Theorem $\ref{gammacongruence}$ which asserts that 
\[\Phi^{\ddag}_{E_1, \Sigma_1}\times \chi_t^{\ddag}(\Gamma, E_1)=1\text{ if and only if }\Phi^{\ddag}_{E_2, \Sigma_1}\times \chi_t^{\ddag}(\Gamma, E_2)=1.\]The conductor $N_1$ of $E_1$ (resp. $N_2$ of $E_2$) is $66=2\times 3\times 11$ (resp. $462=2\times 3\times 7\times 11$). The set $\Sigma_0$ consists of the primes $v\neq 5$ at which $E_1$ or $E_2$ has bad reduction. We have that $\Sigma_0=\{2,3,7,11\}$. We calculate the optimal set of primes $\Sigma_1\subseteq \Sigma_0$. Let $\closure{N}$ denote the prime to $5$ Artin conductor of the residual representation. Recall that $\Sigma_1$ is the set of primes $\Sigma(E_1)\cup \Sigma(E_2)$, where $\Sigma(E_i)$ consists of the primes $v\neq 5$ such that $v|(N_i/\closure{N})$, and if $\mu_5$ is contained in $\Q_v$, then $E$ has split multiplicative reduction at $v$. Note that $\Q_{11}$ contains $\mu_{5}$ and both elliptic curves $E_1$ and $E_2$ have non-split multiplicative reduction at $11$. Hence, $\Sigma_1\subseteq \{2,3,7\}$ is strictly smaller than $\Sigma_0=\{2,3,7,11\}$. Recall that $\Phi_{E,\Sigma_1}$ is the product $\prod_{v\in \Sigma_1} |L_v(E, 1)|_5$. From the discussion in \cite[ p.7]{raysujatha}, $L_v(E_i,1)^{-1}$ is equal to $v^{-1}(l+\beta_v(E_i)-a_l(E))$, where $\beta_v(E_i)$ is $1$ when $E_i$ has good reduction at $v$ and is $0$ otherwise.
These values are calculated from the following:
\[
\begin{split}
    &2+\beta_2(E_1)-a_2(E_1)=3, 3+\beta_3(E_1)-a_3(E_1)=2, 7+\beta_7(E_1)-a_{7}(E_1)=6,\\
     &2+\beta_2(E_2)-a_2(E_2)=3, 3+\beta_3(E_2)-a_3(E_2)=3, 7+\beta_7(E_2)-a_{7}(E_2)=8.\\
\end{split}\]
None of the above values are not divisible by $5$ and hence, $\Phi_{E_1,\Sigma_1}$ and $\Phi_{E_2, \Sigma_1}$ are both $1$. The Euler characteristics $\chi^{\pm}(\Gamma, E_i)$ may be calculated explicitly. The Euler characteristic $\chi(\Gamma, E_i)$ is given up to $5$-adic unit by
\[\chi^{\pm}(\Gamma, E_i)\sim \# \op{Sel}(E/\Q) \times \prod c_v(E_i),\]see \cite[Theorem 1.2]{BDKaust}. For both elliptic curves, $\#\Sh(E_i/\Q)[5]$ and $E_i(\Q)$ has no $5$-torsion. It follows that $\#\op{Sel}(E/\Q)\sim 1$. Furthermore, the Tamagawa products $\prod c_v(E_1)=6$ and $\prod c_v(E_1)=16$. It thus follows that both Euler characteristics $\chi^{\pm}(\Gamma, E_i)=1$ for $i=1,2$. We have thus demonstrated the relationship between Euler characteristics via explicit computation.
\subsection{Example 2}\label{example2}
\par We work out an example over the field $F:=\Q(i)$. Let $E_1$ (resp. $E_2$) be the curves $38a1$ (resp. $114b1$), base-changed to $F$. The elliptic curves $E_1$ and $E_2$ are $5$-congruent. The prime $p=5$ splits into $\mathfrak{p}\mathfrak{p}^*$, where $\mathfrak{p}=(2+i)$ and $\mathfrak{p}^*=(2-i)$. Both curves $E_1$ and $E_2$ are supersingular at the primes $\mathfrak{p}$ and $\mathfrak{p}^*$. The Mordell Weil ranks of $E_1$ (resp. $E_2$) over $F$ are $0$ (resp. $1$). Let $\ddag=(\ddag_1, \ddag_2)$ be a signed vector. It follows that the order of vanishing of the signed Selmer group at $T=0$ is given by $r_{E_1}^{\ddag}=0$ and $r_{E_2}^{\ddag}=1$. Since $r_{E_2}^{\ddag}\leq 1$, it follows from Lemma $\ref{truncdefined}$ that the truncated Euler characteristic $\chi_t(\Gamma, E_2)$ is well defined. On the other hand, since $r_{E_1}^{\ddag}=0$, the truncated Euler characteristic $\chi_t(\Gamma, E_1)$ is well defined and coincides with the usual Euler characteristic $\chi(\Gamma, E_1)$. Since $r_{E_1}^{\ddag}<r_{E_2}^{\ddag}$, Theorem $\ref{gammacongruence}$ asserts that $5$ divides $\Phi_{E_1, \Sigma_1}\times \chi(\Gamma, E_1)$. We demonstrate this via direct calculation by showing that $\Phi_{E_1, \Sigma_1}$ is divisible by $5$.
\par

We shall assume that the Selmer groups $\op{Sel}^{\ddag}(E_i/F^{\op{cyc}})$ are cotorsion. The optimal set of primes $\Sigma_1$ is contained in $\{(i+1), 3, 19\}$, the set of primes $v\nmid 5$ at which either $E_1$ or $E_2$ has bad reduction. Both curves $E_1$ and $E_2$ have split multiplicative reduction at $19$, and hence $19\in \Sigma_1$. The prime $19$ is inert in $F$ and its norm is $19^2=361$. The curve $E_1$ has split multiplicative reduction at $19$, hence $a_{19}(E_1)=1$. Recall that $L_{19}(E_1, 1)^{-1}$ is equal to $P_{19}(E_1,19^{-2})$, see the discussion following $\eqref{pvdefinition}$. The characteristic polynomial $P_{19}(E_1,T)$ is equal to $1-a_{19}(E_1)T$. Therefore,  $L_{19}(E_1, 1)^{-1}=\frac{19^2-a_{19}(E_1)}{19^2}=\frac{360}{361}$. This implies that $|L_{19}(E_1,1)|_{5}=5$. Since $19\in \Sigma_{1}$, it follows that $\Phi_{E_1, \Sigma_1}$ is divisible by $5$.
\par Let $F_1$ and $F_2$ be the maximal pro-$p$ abelian extension of $F$ unramified outside $\mathfrak{p}$, resp. $\mathfrak{p}^*$. Both $F_1$ and $F_2$ are $\Z_p$-extensions of $F$. Let $\mathcal{F}_{\infty}$ be the composite of $F_1$ and $F_2$. We discuss results for truncated Euler characteristics over $\mathcal{F}_{\infty}$. Assume that for $i=1,2$, the Selmer group $\op{Sel}^{\ddag}(E_i/\mathcal{F}_{\infty})$ is in $\mathfrak{M}_H(\op{G})$. Recall that $\op{G}=\op{Gal}(\mathcal{F}_{\infty}/F)$ and $H=\op{Gal}(\mathcal{F}_{\infty}/F^{\op{cyc}})$. The conditions of Hypothesis $\eqref{Finfhyp}$ are satisfied for $\mathcal{F}_{\infty}$. Corollary $\ref{lastcor}$ asserts that $\chi_t(\op{G}, E_i)$ is equal to $\chi_t(\op{G}, E_i)$ for $i=1,2$.

\end{document}